\documentclass[a4paper]{amsart}
\usepackage{amssymb}
\usepackage[mathcal]{euscript}
\usepackage[cmtip,all]{xy}

\newcommand{\bydef}{:=}

\newcommand{\id}{\mathrm{id}}

\DeclareMathOperator*{\ot}{\otimes}

\newcommand{\cA}{\mathcal{A}}
\newcommand{\cB}{\mathcal{B}}

\newcommand{\cJ}{\mathcal{J}}

\newcommand{\cL}{\mathcal{L}}

\newcommand{\cR}{\mathcal{R}}
\newcommand{\cS}{\mathcal{S}}

\newcommand{\cU}{\mathcal{U}}
\newcommand{\cV}{\mathcal{V}}
\newcommand{\cW}{\mathcal{W}}

\newcommand{\ZZ}{\mathbb{Z}}

\newcommand{\CC}{\mathbb{C}}

\newcommand{\FF}{\mathbb{F}}

\newcommand{\chr}[1]{\mathrm{char}\,#1}

\DeclareMathOperator{\End}{\mathrm{End}}
\DeclareMathOperator{\Alg}{\mathrm{Alg}}

\DeclareMathOperator{\Aut}{\mathrm{Aut}}
\DeclareMathOperator{\AAut}{\mathbf{Aut}}

\DeclareMathOperator{\supp}{\mathrm{Supp}}

\newcommand{\frsl}{{\mathfrak{sl}}}

\newcommand{\subo}{_{\bar 0}}
\newcommand{\subuno}{_{\bar 1}}

\DeclareMathOperator{\DDiag}{\mathbf{Diag}}
\DeclareMathOperator{\Mult}{\mathrm{Mult}}

\newtheorem{theorem}{Theorem}[section]
\newtheorem{proposition}[theorem]{Proposition}
\newtheorem{lemma}[theorem]{Lemma}
\newtheorem{corollary}[theorem]{Corollary}

\theoremstyle{definition}
\newtheorem{df}[theorem]{Definition}
\newtheorem{example}[theorem]{Example}

\theoremstyle{remark}
\newtheorem{remark}[theorem]{Remark}

\begin{document}

\title{Gradings on semisimple algebras}

\author[A.S.~C\'ordova-Mart{\'\i}nez]{Alejandra S.~C\'ordova-Mart{\'\i}nez${}^\star$}
\address{Departamento de Matem\'{a}ticas
 e Instituto Universitario de Matem\'aticas y Aplicaciones,
 Universidad de Zaragoza, 50009 Zaragoza, Spain}
\email{sarina.cordova@gmail.com}
\thanks{${}^\star$ 
Supported by grants MTM2017-83506-C2-1-P (AEI/FEDER, UE) and E22\_17R (Grupo de referencia \'Algebra y Geometr{\'\i}a, Diputaci\'on General de Arag\'on).\quad
A.S.~C\'ordova-Mart{\'\i}nez also acknowledges support from the Consejo Nacional de Ciencia y Tecnolog{\'\i}a
(CONACyT, M\'exico) through grant 420842/262964.}

\author[A. Elduque]{Alberto Elduque${}^\star\,{}^\dagger$}
\address{Departamento de Matem\'{a}ticas
 e Instituto Universitario de Matem\'aticas y Aplicaciones,
 Universidad de Zaragoza, 50009 Zaragoza, Spain}
\email{elduque@unizar.es}
\thanks{${}^\dagger$
Corresponding author}



\date{}

\begin{abstract}
The classification of gradings by abelian groups on finite direct sums of simple finite-dimensional nonassociative algebras over an algebraically closed field is reduced, by means of the use of loop algebras, to the corresponding problem for simple algebras. This requires a good definition of (free) products of group-gradings.
\end{abstract}

\maketitle

\section{Introduction}\label{se:intro}

In 1989, Patera and Zassenhaus \cite{PZ89} undertook a systematic study of gradings by abelian groups on finite-dimensional simple Lie algebras over the complex numbers, with fine gradings as the central objects. A key example of fine grading is the root space decomposition of a finite-dimensional semisimple Lie algebra relative to a Cartan subalgebra, but there are many other fine gradings that reflect the symmetries of these algebras. A description of fine gradings on the classical simple Lie algebras (other than $D_4$, which is exceptional in many aspects) over $\CC$ followed in \cite{HPP98}. The classification of fine gradings on all finite-dimensional simple Lie algebras over an algebraically closed field has been recently completed through the efforts of many authors: see \cite{EKmon,YuExc,Eld16}. A survey of the main ideas and results appears in \cite{DE}.

Many of the gradings that appear are related to gradings on simple associative or Jordan algebras, and the
gradings by abelian groups on finite-dimensional simple Lie, associative, or Jordan algebras have also been classified in the last years (see \cite{EKmon} and the references therein). If the grading group becomes part of the definition of the grading, then the right classification is the classification of $G$-gradings up to isomorphism. On the other hand, any grading on a finite-dimensional algebra is a coarsening of a fine grading, so fine gradings, like the root space decomposition mentioned above, become the key objects of study and of classification up to equivalence.

Time is ripe to try to extend the known classifications on simple algebras to semisimple algebras. Usually, the word `semisimple' refers to a given `radical' being zero.  However, in many cases of interest, like finite-dimensional Lie algebras over fields of characteristic zero, or finite-dimensional associative or Jordan algebras, the semisimple algebras are just the finite direct sums of simple algebras. In this paper, we will take this restricted definition of `semisimple algebras'. 

Therefore, in what follows, a \emph{semisimple algebra} will refer to a finite direct sum of simple algebras. The word `algebra' will refer to a vector space over a ground field $\FF$, endowed with a bilinear multiplication. No other restriction will be imposed, so our results apply to any variety of nonassociative (i.e., not necessarily associative) algebras.

The goal of this work is to show that the classification of gradings on semisimple algebras can be reduced, through the use of loop algebras, studied thoroughly in \cite{loop}, to the known classifications for simple algebras. This reduction is far from trivial, and requires some preparation, that includes a good definition of product gradings, which allow us to move from gradings on some algebras, to gradings on their direct sum.

In a recent paper \cite{CDE}, the study of gradings (and the affine group scheme of automorphisms) on the $10$-dimensional Kac's Jordan superalgebra was reduced to the study of gradings on the direct sum of two copies of the `tiny' ($3$-dimensional) Kaplansky superalgebra. In this case the dimension is small enough so that some `ad hoc' arguments allow a complete classification.
On the other hand, the first author has been able to reduce, in work in preparation, the classification of gradings on tensor products of Cayley algebras to the corresponding problem on cartesian products of Cayley algebras, where the results on this paper apply to reduce the problem to gradings on Cayley algebras, which are well understood \cite{Eld98}. The tensor products of two Cayley algebras are examples of simple structurable algebras, that can be used to give some nice explicit constructions of the exceptional simple Lie algebras.

\smallskip

The next section will present the background needed on gradings. Some standard definitions will be slightly changed, in order to give a precise definition of product gradings. We are interested in gradings by abelian groups, but will also look to more general gradings, for which the notion of product grading is quite natural (Definition \ref{df:product_grading}). Any general grading has a finest coarsening which is a group-grading, with the same universal group. This is the clue to define the free product group-grading in Definition \ref{df:free_product_grading}.

Section 3 will be devoted to review the definition and some properties of loop algebras, defined and studied in \cite{loop}. The universal groups of the gradings on these algebras will be determined. These results are used in Section 4 to classify $G$-gradings, up to isomorphism, on finite-dimensional semisimple algebras over an algebraically closed field.

The final section 5 is devoted to fine group-gradings on semisimple algebras. In general, fine group-gradings on a semisimple algebra are equivalent to free product group-gradings of fine gradings on semisimple and graded-simple algebras, but some restrictions apply. Corollary \ref{co:fine_semisimple} gives the classification, up to equivalence, of the fine group-gradings on finite-dimensional semisimple algebras over an algebraically closed field.

\smallskip

\section{Background on gradings}\label{se:gradings}

This section is devoted to a review of the main concepts on gradings on algebras.
The general background on gradings can be found in the monograph \cite{EKmon}. However some slight variations of the main definitions will be performed here, with a view towards a right definition of \emph{product gradings}, which will be instrumental in classifying fine gradings, up to equivalence, on semisimple algebras.

\medskip

\subsection{Gradings}\label{ss:gradings}
We start with a very general definition of grading.

\begin{df}\label{df:grading}
A \emph{grading} on an algebra $\cA$ over a field $\FF$ is a set $\Gamma$ of nonzero subspaces of $\cA$ such that $\cA=\bigoplus_{\cU\in \Gamma}\cU$ and for any $\cU,\cV\in\Gamma$, there is $\cW\in \Gamma$ such that $\cU\cV\subseteq \cW$.
\end{df}

There are several natural related notions in
 the situation of Definition \ref{df:grading}: 
 
\begin{itemize}

\item The pair $(\cA,\Gamma)$ is said to be a \emph{graded algebra}.

\item The elements of $\Gamma$ are called the \emph{homogeneous components}. The nonzero elements of the homogeneous components are called \emph{homogeneous elements}.

\item A subalgebra $\cS$ of $\cA$ is called a \emph{graded subalgebra} if $\cS=\bigoplus_{\cU\in\Gamma}\cU\cap\cS$. In this case 
\begin{equation}\label{eq:Gamma_S}
\Gamma\vert_\cS\bydef \{\cU\cap\cS\mid \cU\in \Gamma\ \text{and}\ 0\neq \cU\cap\cS\}
\end{equation}
is the \emph{induced grading} on $\cS$. A \emph{graded ideal} is an ideal which is, at the same time, a graded subalgebra.

\item The graded algebra $(\cA,\Gamma)$ is said to be \emph{graded-simple} if $\cA$ does not contain any proper graded ideal and $\cA^2\neq 0$. In this case $\cA^2=\cA$, as $\cA^2$ is always a graded ideal.

\item Given another grading $\Gamma'$ on $\cA$, $\Gamma$ is said to be a \emph{refinement} of $\Gamma'$ (and $\Gamma'$ a \emph{coarsening} of $\Gamma$) if any subspace $\cU\in\Gamma$ is contained in a subspace in $\Gamma'$. If at least one of these containments is proper, then the refinement is said to be proper. In this situation $\Gamma$ is said to be \emph{finer} than $\Gamma'$, and $\Gamma'$ \emph{coarser} than $\Gamma$, and the map $\pi:\Gamma\rightarrow \Gamma'$, that sends any $\cU\in \Gamma$ to the element $\cU'\in\Gamma'$ that contains it, is a surjection. The refinement is proper if, and only if, $\pi$ is not a bijection.

\item The grading $\Gamma$ is said to be \emph{fine} if it admits no proper refinements. 

Any grading on a finite-dimensional algebra is a coarsening of a fine grading.

\item Given two graded algebras $(\cA,\Gamma)$ and $(\cA',\Gamma')$, an \emph{equivalence} $\varphi:(\cA,\Gamma)\rightarrow (\cA',\Gamma')$ is an isomorphism $\varphi:\cA\rightarrow \cA'$ such that $\varphi(\cU)\in\Gamma'$ for each $\cU\in\Gamma$.

\end{itemize}

Given a graded algebra $(\cA,\Gamma)$, consider the abelian group $U(\Gamma)$ generated by the set $\Gamma$, subject to the relations $\cU\cV\cW^{-1}=e$ ($e$ denotes the neutral element) for each triple $\cU,\cV,\cW$ in $\Gamma$ such that $0\neq \cU\cV\subseteq \cW$:
\[
U(\Gamma)\bydef\langle \Gamma\mid \cU\cV\cW^{-1}=e\ \text{if}\ 0\neq \cU\cV\subseteq \cW\rangle.
\]
That is, $U(\Gamma)$ is the quotient of the free abelian group generated by $\Gamma$, modulo the normal subgroup generated by the elements $\cU\cV\cW^{-1}$ above. Consider also the natural map:
\[
\begin{split}
\delta_\Gamma^U:\Gamma&\longrightarrow U(\Gamma)\\
  \cU\,&\mapsto\ [\cU],
\end{split}
\]
where $[\cU]$ denotes the class of $\cU$ in $U(\Gamma)$.

\begin{df}\label{df:universal_group}
The pair $\bigl(U(\Gamma),\delta_\Gamma^U\bigr)$ is called the \emph{universal group} of the grading $\Gamma$.
\end{df}

There is a natural notion of homomorphism between pairs $(G,\delta)$, for a group $G$ and a map $\delta:\Gamma\rightarrow G$. Given two such pairs $(G^i,\delta^i)$, $i=1,2$, a homomorphism is a group homomorphism $\varphi:G^1\rightarrow G^2$ such that the diagram
\[
\xymatrix{
&& G^1\ar[dd]^{\varphi}\\
\Gamma\ar[rru]^{\delta^1}\ar[rrd]^{\delta^2} &\\
&& G^2
}
\]
is commutative.

\begin{example}\label{ex:trivial_grading}
The \emph{trivial grading} on a nonzero algebra $\cA$ is the grading $\Gamma=\{\cA\}$. We have two possibilities for its universal group:
\begin{itemize}
\item if $\cA^2\neq 0$, then $U(\Gamma)$ is the trivial group.
\item if $\cA^2=0$, then $U(\Gamma)$ is the infinite cyclic group (isomorphic to $\ZZ$).
\end{itemize}
\end{example}

\begin{example}\label{ex:FxF}
Consider the cartesian product $\cA=\FF\times\FF=\FF e_1\oplus\FF e_2$, for the orthogonal idempotents $e_1=(1,0)$ and $e_2=(0,1)$. Then $\Gamma=\{\FF e_1,\FF e_2\}$ is a grading. Denote $u_i\bydef \FF e_i$, $i=1,2$. Then
\[
U(\Gamma)=\langle u_1,u_2\mid u_1^2=u_1,\ u_2^2=u_2\rangle=\{e\}
\]
is the trivial group, even though our grading $\Gamma$ is not trivial.
\end{example}

The next result is a direct consequence of the definitions:

\begin{proposition}\label{pr:equivalence}
Let $\varphi:(\cA,\Gamma)\rightarrow (\cA',\Gamma')$ be an equivalence, then $\varphi$ induces a bijection $\alpha_{\varphi}:\Gamma\rightarrow \Gamma'$ given by $\alpha_{\varphi}(\cU)=\varphi(\cU)$, which in turn induces a group isomorphism $\alpha^U_{\varphi}:U(\Gamma)\rightarrow U(\Gamma')$, such that the diagram
\[
\xymatrix{
\Gamma\ar[r]^{\delta_\Gamma^U}\ar[d]_{\alpha_{\varphi}} & U(\Gamma)\ar[d]^{\alpha^U_{\varphi}}\\
\Gamma'\ar[r]^{\delta_{\Gamma'}^{U}} & U(\Gamma')
}
\]
is commutative.
\end{proposition}

The universal group is strongly related to the group of \emph{diagonal automorphisms} in the affine group schemes sense.

Recall that an \emph{affine group scheme} over $\FF$ is a representable functor from the category $\Alg_\FF$ of 
unital associative commutative $\FF$-algebras (not necessarily of finite dimension) to the category of groups (see, for instance, \cite{Wh}, \cite[Chapter VI]{KMRT} or \cite[Appendix A]{EKmon}). 

For example, the automorphism group scheme $\AAut_\FF(\cA)$ of a finite-dimensional algebra $\cA$ is defined by 
$\AAut_\FF(\cA)(R)=\Aut_R(\cA\ot_\FF R)$ for every object $R$ in $\Alg_\FF$. (This  may be strictly larger than 
the affine group scheme corresponding to the algebraic group $\Aut_\FF(\cA)$.)

Given a finite-dimensional graded algebra $(\cA,\Gamma)$, the \emph{diagonal group scheme} of $\Gamma$, denoted by $\DDiag_\Gamma(\cA)$, is the subgroup scheme of $\AAut_\FF(\cA)$ whose $R$-points consist of the automorphisms of $\cA\otimes_\FF R$ which act on each $\cU\otimes_\FF R$, $\cU\in\Gamma$, by multiplication by a scalar $r_\cU\in R^\times$ ($R^\times$ denotes the group of invertible elements in $R$).
Thus, any $R$-point of $\DDiag_\Gamma(\cA)$ is determined by a collection of elements $r_\cU\in R^\times$, $\cU\in\Gamma$, that must satisfy $r_\cU r_\cV=r_\cW$ if $0\neq \cU\cV\subseteq \cW$. This proves the next result. (See \cite[\S 1.4]{EKmon}.)

Given any finitely generated abelian group $G$, its \emph{Cartier dual} $G^D$ is the affine group scheme represented by the group algebra $\FF G$. Its $R$-points consist of the algebra homomorphisms $\FF G\rightarrow R$, which can be identified with the group homomorphisms $G\rightarrow R^\times$.  The affine group schemes isomorphic to $G^D$ for a finitely generated abelian group are called  \emph{diagonalizable}.

\begin{theorem}\label{th:diag_universal}
Let $(\cA,\Gamma)$ be a finite-dimensional graded algebra. Then the affine group scheme $\DDiag_\Gamma(\cA)$ is naturally isomorphic to $U(\Gamma)^D$.
\end{theorem}

\medskip

\subsection{Group-gradings}\label{ss:group_gradings}
We are mainly interested in gradings by abelian groups:

\begin{df}\label{df:G_grading}
Given an abelian group $G$, a $G$-grading on an algebra $\cA$ is a triple $(\Gamma,G,\delta)$, where $\Gamma$ is a grading on $\cA$, and $\delta:\Gamma\rightarrow G$ is a one-to-one map, such that for any $\cU,\cV,\cW\in\Gamma$ such that $0\neq \cU\cV\subseteq\cW$, $\delta(\cU)\delta(\cV)=\delta(\cW)$.
\end{df}

Given a $G$-grading $(\Gamma,G,\delta)$, define $\cA_g=\cU$ if $\delta(\cU)=g$, and define  $\cA_g=0$ if $g$ is not in the range of $\delta$. Then $\cA=\bigoplus_{g\in G}\cA_g$, and we recover the usual expression for a $G$-grading. The map $\delta$ is called the \emph{degree map}. For $\cA_g\neq 0$, $\delta(\cA_g)$ is simply $g$. We also write $\deg(x)=g$ for any $0\neq x\in \cA_g$.

As shown in \cite[Proposition 1.36]{EKmon}, a $G$-grading on an algebra $\cA$ corresponds to a homomorphism of affine group schemes $G^D\rightarrow \AAut_\FF(\cA)$. For every object $R$ in $\Alg_\FF$, the image of a group homomorphism $f:G\rightarrow R^\times$ (i.e., an $R$-point of the Cartier dual $G^D$), is the automorphism of $\cA\otimes_\FF R$ such that $x_g\otimes r\mapsto x_g\otimes f(g)r$, for any $x_g\in \cA_g$ and $r\in R$.

As for general gradings, there are several natural related notions in
 the situation of Definition \ref{df:G_grading}: 
 
\begin{itemize}

\item The $4$-tuple $(\cA,\Gamma,G,\delta)$ is said to be a \emph{$G$-graded algebra}. If the other components are clear from the context, we may refer simply to a $G$-graded algebra $\cA$.

\item The range of $\delta$ is the subset $\supp_G(\Gamma)\bydef\{g\in G: \cA_g\neq 0\}$, which is called the \emph{support} of the $G$-grading. Thus $\Gamma=\{ \cA_g\mid g\in\supp_G(\Gamma)\}$.

\item Given a $G$-graded algebra $(\cA,\Gamma,G,\delta)$, any graded subalgebra $\cS$ of the graded algebra $(\cA,\Gamma)$ gives rise to the $G$-graded algebra $(\cS,\Gamma\vert_\cS,G,\delta\vert_\cS)$,  $\Gamma\vert_\cS$ as in \eqref{eq:Gamma_S}, with $\delta\vert_\cS(\cU\cap\cS)=\delta(\cU)$. When referring to a $G$-graded subalgebra of $(\cA,\Gamma,G,\delta)$ we will mean a graded subalgebra $\cS$ of $(\cA,\Gamma)$, endowed with the $G$-grading above. The same applies to $G$-graded ideals.

\item Given two $G$-graded algebras $(\cA,\Gamma,G,\delta)$ and $(\cA',\Gamma',G,\delta')$, an \emph{isomorphism} $\varphi:(\cA,\Gamma,G,\delta)\rightarrow (\cA',\Gamma',G,\delta')$ is an isomorphism $\varphi:\cA\rightarrow \cA'$ such that $\varphi(\cA_g)=\cA'_g$ for each $g\in G$.

\item Given a $G$-graded algebra $(\cA,\Gamma,G,\delta)$ and an abelian group $H$, any group homomorphism $\beta:G\rightarrow H$ defines an $H$-grading $(\cA,\Gamma',H,\delta')$ by $\cA=\bigoplus_{h\in H}\cA'_h$, with $\cA'_h\bydef \bigoplus_{\beta(g)=h}\cA_g$ for any $h\in H$. The new grading $\Gamma'$ is a coarsening of $\Gamma$. If $\pi:\Gamma\rightarrow \Gamma'$ is the corresponding surjection, then the diagram
\[
\xymatrix{
\Gamma\ar[r]^{\delta}\ar[d]_{\pi} & G\ar[d]^{\beta}\\
\Gamma'\ar[r]^{\delta'} & H
}
\]
is commutative.

In this case, the grading $(\Gamma',H,\delta')$ is said to be the \emph{coarsening of $(\Gamma,G,\delta)$ induced by $\beta$}.

\end{itemize}

\begin{df}\label{df:group-grading}
A grading $\Gamma$ on an algebra $\cA$ is called a \emph{group-grading} if there is an abelian group $G$ and a $G$-grading of the form $(\Gamma,G,\delta)$ (i.e., the first component of the $G$-grading is $\Gamma$).
\end{df}

In this situation, we say that $\Gamma$ can be realized as a $G$-grading, or by the $G$-grading $(\Gamma,G,\delta)$, and we will talk about the group-graded algebra $(\cA,\Gamma)$.

A group-grading $\Gamma$ on an algebra $\cA$ is said to be a \emph{fine group-grading} if it admits no proper refinements in the class of group gradings. Any group-grading on a finite-dimensional algebra is a coarsening of a fine group-grading.

\begin{remark}\label{rk:fine_group_grading}
If $\chr\FF=2$, then the algebra $\FF\times\FF$  admits a unique group-grading: the trivial one. Thus the trivial grading is a fine group-grading, but it is not a fine grading, because the grading considered in Example \ref{ex:FxF} is finer. Note that $\AAut_\FF(\FF\times\FF)=\mathsf{C}_2$, the constant group scheme corresponding to the cyclic group of order $2$: $C_2$, which is not diagonalizable because $\chr\FF=2$.

 (Note that if $\chr\FF\neq 2$, the trivial grading can be refined to the grading by $\ZZ/2$ with $(\FF\times\FF)\subo=\FF(1,1)$ and $(\FF\times\FF)\subuno=\FF(1,-1)$.) 
\end{remark}

The next result, whose proof is straightforward, characterizes group-gradings and explains the adjective \emph{universal} in the definition of the universal group:

\begin{theorem}\label{th:universal}
Let $(\cA,\Gamma)$ be a graded algebra, with universal group $\bigl(U(\Gamma),\delta_\Gamma^U\bigr)$. Then $\Gamma$ is a group-grading if and only if $\delta_\Gamma^U$ is one-to-one. In this case $(\cA,\Gamma,U(\Gamma),\delta_\Gamma^U)$ is a $U(\Gamma)$-graded algebra.

Moreover, if $\Gamma$ can be realized by the $G$-grading $(\Gamma,G,\delta)$, then there is a unique group homomorphism $\varphi:U(\Gamma)\rightarrow G$, such that the diagram
\[
\xymatrix{
&& U(\Gamma)\ar[dd]^{\varphi}\\
\Gamma\ar[rru]^-{\delta_\Gamma^U}\ar[rrd]^-{\delta} &\\
&& G
}
\]
is commutative. (In other words, there is a unique homomorphism $\bigl(U(\Gamma),\delta_\Gamma^U)\rightarrow (G,\delta)$.)
\end{theorem}

We end this subsection with a trivial but useful remark:

\begin{proposition}\label{pr:equivalence_iso}
Let $(\cA^i,\Gamma^i)$, $i=1,2$, be two group-graded algebras with universal groups $\bigl(U(\Gamma^i),\delta_{\Gamma^i}^U\bigr)$, and let $\varphi:(\cA^1,\Gamma^1)\rightarrow (\cA^2,\Gamma^2)$ be an equivalence. Then with $U=U(\Gamma^2)$ and $\delta^1:\Gamma^1\rightarrow U$ given by $\delta^1(\cU)=\delta_{\Gamma^2}^U\bigl(\varphi(\cU)\bigr)$ for $\cU \in \Gamma^1$, the equivalence $\varphi$ becomes an isomorphism $\varphi:(\cA^1,\Gamma^1,U,\delta^1)\rightarrow (\cA^2,\Gamma^2,U,\delta_{\Gamma^2}^U)$ of $U$-graded algebras.
\end{proposition}

\medskip

\subsection{The group grading induced by a grading}

Given any grading $\Gamma$, there is always a natural group-grading attached to it.

\begin{df}\label{df:group_grading_induced}
Let $\Gamma$ be a grading on the algebra $\cA$, and let $\bigl(U(\Gamma),\delta_\Gamma^U\bigr)$ be its universal group. The coarsening $\Gamma_{\mathrm{gr}}$ defined by
\[
\Gamma_{\mathrm{gr}}\bydef\left\{\sum_{\delta_\Gamma^U(\cU)=u}\cU\quad \Big\vert\quad  u\in \delta_\Gamma^U(\Gamma)\right\}
\]
is called the \emph{group-grading induced by} $\Gamma$. The grading $\Gamma_{\mathrm{gr}}$ can be realized by the $U(\Gamma)$-grading $\bigl(\Gamma_{\mathrm{gr}},U(\Gamma),\delta_{\Gamma_{\mathrm{gr}}}^U\bigr)$, where 
\[
\delta_{\Gamma_{\mathrm{gr}}}^U\Bigl(\sum_{\delta_\Gamma^U(\cU)=u}\cU\Bigr)=u
\]
for any $u\in \delta_\Gamma^U(\Gamma)$.
\end{df}

Theorem \ref{th:universal} implies our next result:

\begin{theorem}\label{th:grading_group-grading}
Let $\Gamma$ be a grading on an algebra $\cA$ with universal group $\bigl(U(\Gamma),\delta_\Gamma^U\bigr)$. 
\begin{itemize} 
\item $\Gamma$ is a group-grading if and only if $\Gamma=\Gamma_{\mathrm{gr}}$.
(Hence $\delta_\Gamma^U=\delta_{\Gamma_{\mathrm{gr}}}^U$.)

\item For any abelian group $G$, and any $G$-grading $(\Gamma',G,\delta')$ such that $\Gamma'$ is a coarsening of $\Gamma$, there is a unique group homomorphism $\beta:U(\Gamma)\rightarrow G$ such that $(\Gamma',G,\delta')$ is the coarsening of $(\Gamma_{\mathrm{gr}},U(\Gamma),\delta_{\Gamma_{\mathrm{gr}}}^U)$ induced by $\beta$. In particular, $(U(\Gamma),\delta_{\Gamma_{\mathrm{gr}}}^U)$ is, up to isomorphism, the universal group of $\Gamma_{\mathrm{gr}}$.
\end{itemize}
\end{theorem}
\begin{proof} 
The first part is clear. As for the second part, since $\Gamma'$ is a coarsening of $\Gamma$, there is the attached surjective map $\pi:\Gamma\rightarrow \Gamma'$. The composite map $\delta'\circ \pi:\Gamma\rightarrow G$ induces a unique group homomorphism from the free abelian group generated by $\Gamma$ into $G$, which factors through $U(\Gamma)$, thus giving the desired group homomorphism $\beta:U(\Gamma)\rightarrow G$.
\end{proof}

 In particular, the grading in Example \ref{ex:FxF} is not a group-grading, because $\Gamma_{\mathrm{gr}}=\{\cA\}$.

\begin{example}\label{ex:non-semigroup.grading}
Consider the special linear Lie algebra of degree $2$:
\[ 
\frsl_2 = \left\langle  E= \left(\begin{smallmatrix}
0 & 1 \\
0 & 0 
\end{smallmatrix}\right),
F= \left(\begin{smallmatrix}
0 & 0 \\
1 & 0 
\end{smallmatrix}\right),
H= \allowbreak \left(\begin{smallmatrix}
1 & 0 \\
0 & -1 
\end{smallmatrix}\right) \right\rangle,
\] 
over a ground field $\FF$ of characteristic not $2$. This is a simple Lie algebra, its bracket is determined by: 
\begin{equation}\label{eq:EFH}
[E,F]= H,\quad [H,F]=-2F,\quad\text{and}\quad [H,E]=2E.
\end{equation}
On the cartesian product $\cL=\frsl_2\times\frsl_2$ of two copies of $\frsl_2$, consider the following grading introduced in \cite{Eld09}:
\[
\Gamma=\lbrace \FF H \times \FF H,\; \FF E \times 0,\; \FF F \times 0,\; 0 \times \FF (E+F),\; 0 \times \FF (E-F)\rbrace.
\]
It is straightforward to check that it is a grading, but not a group-grading. 

Then the universal group $U(\Gamma)$ is generated by $\alpha\bydef\FF H\times\FF H$, $\beta=\FF E \times 0$, $\gamma=\FF F \times 0$, $\delta=0 \times \FF (E+F)$, and $\epsilon= 0 \times \FF (E-F)$, subject to the following conditions obtained from \eqref{eq:EFH}:
\[
\alpha\beta=\beta,\quad \alpha\gamma=\gamma,\quad \alpha\delta=\epsilon,\quad \alpha\epsilon=\delta,\quad \beta\gamma=\alpha,\quad \delta\epsilon=\alpha,
\]
so that $\alpha=e$ (the neutral element), $\beta=\gamma^{-1}$ and $\epsilon=\delta=\delta^{-1}$. Therefore $U(\Gamma)$ is generated by (the classes of) $\beta$ and $\delta$ and it is isomorphic to $\ZZ\times \left(\ZZ/2\right)$. Moreover, the induced group-grading is
\[
\Gamma_{\mathrm{gr}}=\left\{ \cL_e=\FF H\times\FF H,\; \cL_\beta=\FF E\times 0,\; \cL_{\beta^{-1}}=\FF F\times 0,\; \cL_\delta=0\times (\FF E\oplus\FF F)\right\}.
\]
Later on we will come back to this semisimple algebra $\cL=\frsl_2\times\frsl_2$.\qquad\qed
\end{example}

Equivalence of gradings is clearly inherited by the induced group-gradings:

\begin{theorem}\label{th:equivalence_groupgradings}
If $\varphi:(\cA,\Gamma)\rightarrow (\cA',\Gamma')$ is an equivalence of graded algebras, then $\varphi$ is also an equivalence $(\cA,\Gamma_{\mathrm{gr}})\rightarrow (\cA',\Gamma'_{\mathrm{gr}})$.
\end{theorem}

\smallskip

\subsection{Product gradings} 
\null\quad

\begin{df}\label{df:product_grading}
Let $(\cA^i,\Gamma^i)$ be a graded $\FF$-algebra, $i=1,\ldots,n$. The grading on $\cA^1\times\cdots\times \cA^n$ given by:
\[
\Gamma^1\times\cdots\times\Gamma^n\bydef\bigcup_{i=1}^n\left\{0\times\cdots\times\cU\times\cdots\times 0\mid \cU\in\Gamma^i\right\}
\]
is called the \emph{product grading} of the $\Gamma^i$'s.
\end{df}

The universal group of the product grading is  easily seen to be
\begin{equation}\label{eq:product_grading}
\bigl(U(\Gamma^1\times\cdots\times\Gamma^n), \delta^U_{\Gamma^1\times\cdots\times\Gamma^n}\bigr),
\end{equation}
given by the following formulas:
\[
\begin{split}
&U(\Gamma^1\times\cdots\times\Gamma^n)=U(\Gamma^1)\times\cdots\times U(\Gamma^n),\\
&\delta^U_{\Gamma^1\times\cdots\times\Gamma^n}\bigl(0\times\cdots\times \cU\times\cdots\times 0)=
   \bigl(e,\cdots,\delta^U_{\Gamma^i}(\cU),\cdots,e\bigr)\ \forall\cU\in\Gamma^i,\ \forall i=1,\ldots,n.
\end{split}
\]

\begin{example}\label{ex:FxFbis}
The grading in Example \ref{ex:FxF}, which is not a group-grading, is the product grading of the trivial gradings on the two copies of $\FF$.
\end{example}

As the previous example shows, even if $\Gamma^1,\ldots,\Gamma^n$ are group-gradings, the product grading may fail to be so. Therefore we need a different definition of product grading for group-gradings.

\begin{df}\label{df:product_group_gradings}
Let $G^i$ be an abelian group and
let $(\cA^i,\Gamma^i,G^i,\delta^i)$ be a $G^i$-group-graded algebra, $i=1,\ldots,n$, then the \emph{product group-grading} $(\Gamma^1,G^1,\delta^1)\times\cdots\times (\Gamma^n,G^n,\delta^n)$ is the group-grading on $\cA^1\times\cdots\times\cA^n$ by the abelian group $G^1\times\cdots\times G^n$ with:
\[
\begin{split}
\bigl(\cA^1\times\cdots\times\cA^n\bigr)_{(e,\ldots,e)}&=\cA^1_e\times\cdots\times\cA^n_e,\\
\bigl(\cA^1\times\cdots\times\cA^n\bigr)_{(e,\ldots,g_i,\ldots,e)}&=0\times\cdots\times\cA^i_{g_i}\times\cdots\times 0,\ i=1,\ldots,n,\ e\neq g_i\in G^i\\
\bigl(\cA^1\times\cdots\times\cA^n\bigr)_{(g_1,\ldots,g_n)}&=0,\ \text{if there are at least two indices $1\leq i<j\leq n$}\\[-2pt]
&\hspace*{2.3in}\text{with $g_i\neq e\neq g_j$.}
\end{split}
\]
\end{df}

Our next result shows the naturality of this definition.

\begin{theorem}\label{th:product_group_grading}
Let $\Gamma^i$ be a group-grading on an algebra $\cA^i$, and let $\bigl(U(\Gamma^i),\delta^U_{\Gamma^i}\bigr)$ be its universal group, $i=1,\ldots,n$. Then the product group-grading
\[
\Bigl(\Gamma^1,U(\Gamma^1),\delta^U_{\Gamma^1}\Bigr)\times\cdots\times \Bigl(\Gamma^n,U(\Gamma^n),\delta^U_{\Gamma^n}\Bigr)
\]
coincides with the induced group-grading
\[
\Bigl((\Gamma^1\times\cdots\times\Gamma^n)_{\mathrm{gr}},U(\Gamma^1\times\cdots\times\Gamma^n),\delta^U_{(\Gamma^1\times\cdots\times\Gamma^n)_{\mathrm{gr}}}\Bigr).
\]
(That is, the group-grading induced from the product grading $\Gamma^1\times\cdots\times\Gamma^n$, with its universal grading group.)
\end{theorem}
\begin{proof}
We already know (see Theorem \ref{th:grading_group-grading} and Equation \eqref{eq:product_grading}) that 
\[
U\bigl((\Gamma^1\times\cdots\times\Gamma^n)_{\mathrm{gr}}\bigr) =U(\Gamma^1\times\cdots\times \Gamma^n)
=U(\Gamma^1)\times\cdots\times U(\Gamma^n).
\]
Now everything follows from the definition of the group-grading induced by a grading.
\end{proof}

This result motivates our next definition:

\begin{df}\label{df:free_product_grading}
Let $\Gamma^i$ be a group-grading on an algebra $\cA^i$, $i=1,\ldots,n$. Then the group-grading $(\Gamma^1\times\cdots\times\Gamma^n)_{\mathrm{gr}}$ on $\cA^1\times\cdots\times\cA^n$ is called the \emph{free product group-grading} of the $\Gamma^i$'s.
\end{df}

\begin{corollary}\label{co:product_equivalence}
Let the group-graded algebras $(\cA^i,\Gamma^i)$ and $(\cB^i,\tilde\Gamma^i)$ be equivalent, for $i=1,\ldots,n$. Then so are the group-graded algebras $(\cA^1\times\cdots\times \cA^n,(\Gamma^1\times\cdots\times\Gamma^n)_{\mathrm{gr}})$ and $(\cB^1\times\cdots\times \cB^n,(\tilde\Gamma^1\times\cdots\times\tilde\Gamma^n)_{\mathrm{gr}})$.
\end{corollary}

\begin{example}\label{ex:sl2.product.group-grading}
Over an algebraically closed ground field $\FF$ of characteristic not $2$, consider the simple Lie algebra $\frsl_2$.
Up to equivalence, there are only two fine gradings on $\frsl_2$ (see \cite[Theorem 3.55]{EKmon}):
\begin{itemize}
\item $\Gamma_{\frsl_2}^1$ with universal group $\ZZ$ and homogeneous components:
\[
(\frsl_2)_{-1}=\FF F,\quad (\frsl_2)_{0}=\FF H,\quad (\frsl_2)_{1}=\FF E.
\]

\item $\Gamma_{\frsl_2}^2$ with universal group $\left(\ZZ/2\right)^2$ and homogeneous components:
\begin{equation}\label{eq:Gamma2sl2}
(\frsl_2)_{(\bar 1,\bar 0)}=\FF H,\quad (\frsl_2)_{(\bar0,\bar 1)}=\FF(E+F) ,\quad (\frsl_2)_{(\bar 1,\bar 1)}=\FF (E-F).
\end{equation}
\end{itemize}

The gradings on $\cL=\frsl_2\times\frsl_2$ obtained as free product group-gradings of the fine gradings above are the following:
\begin{itemize}
\item $\left(\Gamma_{\frsl_2}^1\times\Gamma_{\frsl_2}^1\right)_{\mathrm{gr}}$ with universal group $\ZZ\times\ZZ$ and homogeneous components:
\[
\begin{array}{ll}
\cL_{(0,0)}= \FF H \times \FF H,\qquad\null & \\
\cL_{(1,0)}= \FF E \times 0, & \cL_{(0,1)}= 0 \times \FF E, \\
\cL_{(-1,0)}= \FF F \times 0, & \cL_{(0,-1)}= 0 \times \FF F.
\end{array}
\]

\item $\left(\Gamma_{\frsl_2}^1\times\Gamma_{\frsl_2}^2\right)_{\mathrm{gr}}$ with universal group $\ZZ\times\left(\ZZ/2\right)^2$ and homogeneous components:
\[
\begin{array}{ll}
\cL_{(0, (\bar{0}, \bar{0}))}= \FF H \times 0, &\cL_{(0, (1, \bar{0}))}= 0 \times \FF H, \\
\cL_{(1, (\bar{0}, \bar{0}))}= \FF E \times 0, & \cL_{(-1, (\bar{0}, \bar{0}))}= \FF F \times 0,\\
\cL_{(0, (\bar{0}, \bar{1}))}= 0 \times \FF (E+F),\qquad\null &  
\cL_{(0, (\bar{1}, \bar{1}))}= 0 \times \FF (E-F).
\end{array}
\]

\item $\left(\Gamma_{\frsl_2}^2\times\Gamma_{\frsl_2}^2\right)_{\mathrm{gr}}$ with universal group $\left(\ZZ/2\right)^4$ and homogeneous components:
\[
\begin{array}{ll}
\cL_{(\bar{1}, \bar{0},\bar{0},\bar{0})}= \FF H \times 0, &\cL_{(\bar{0}, \bar{0},\bar{1},\bar{0})}= 0 \times \FF H,  \\
\cL_{(\bar{0}, \bar{1},\bar{0},\bar{0})}= \FF (E+F) \times 0,\qquad\null & \cL_{(\bar{0}, \bar{0},\bar{0},\bar{1})}= 0 \times \FF (E+F),   \\
\cL_{(\bar{1}, \bar{1},\bar{0},\bar{0})}= \FF (E-F) \times 0 &\cL_{(\bar{0}, \bar{0},\bar{1},\bar{1})}= 0 \times \FF (E-F).
\end{array}
\]
\end{itemize}
It is checked easily that all these free product gradings are fine group-gradings, but we will prove (see Corollary \ref{co:fine_semisimple} and Example \ref{ex:sl2.cartesian.product}) that they do not exhaust the list of fine group-gradings, up to equivalence, on the semisimple algebra $\cL$.
\end{example}

\smallskip
Besides the product grading, the product group-grading, and the free product group-grading, there is one more natural definition of product of gradings in case the grading group is fixed.

Given an abelian group $G$, and $G$-graded algebras $(\cA^i,\Gamma^i,G,\delta^i)$, $i=1,\ldots,n$, there is a natural $G$-grading $(\Gamma,G,\delta)$ on the cartesian product $\cA^1\times\cdots\times\cA^n$ determined by
\[
(\cA^1\times\cdots\times\cA^n)_g=\cA^1_g\times\cdots\times\cA^n_g
\]
for any $g\in G$.

\begin{df}\label{df:G_product}
The $G$-grading above will be denoted by $\bigl(\cA^1\times\cdots\times\cA^n,\Gamma^1\times_G\cdots\times_G\Gamma^n,G,\delta^1\times_G\cdots\times_G\delta^n\bigr)$ and will be called the \emph{product $G$-grading} of the $(\Gamma^i,G,\delta^i)$'s.
\end{df}

\section{Loop algebras}\label{se:loop}

A key role in understanding gradings on semisimple algebras is played by \emph{loop algebras}, as defined in \cite{loop}.

\begin{df}(\cite[Definition 3.1.1]{loop})\label{df:loop}
Let $\pi:G\rightarrow \overline{G}$ be a surjective group homomorphism between the abelian groups $G$ and $\overline{G}$. Given any $\overline{G}$-graded algebra $(\cA,\overline{\Gamma},\overline{G},\bar\delta)$, the associated \emph{loop algebra} is the $G$-graded algebra $(L_\pi(\cA),\Gamma,G,\delta)$, where 
\[
L_\pi(\cA)\bydef\bigoplus_{g\in G}\cA_{\pi(g)}\otimes g\quad\Bigl(\leq \cA\otimes_\FF\FF G\Bigr)
\]
and $L_\pi(\cA)_g=\cA_{\pi(g)}\otimes g$ for any $g\in G$.
\end{df}

From its own definition, the following result is clear:

\begin{proposition}\label{pr:loop_finite_dimension}
The dimension of the loop algebra $L_\pi(\cA)$ is finite if and only if so is the dimension of $\cA$ and $\ker\pi$ is finite.
\end{proposition}

Recall that given any $G$-graded algebra $(\cA,\Gamma,G,\delta)$, the same grading $\Gamma$ can be realized as a group-grading by different groups and, in particular, by its universal group. Given a surjective group homomorphism $\pi:G\rightarrow\overline{G}$ and a $\overline{G}$-graded algebra, we may consider the universal groups of the gradings in both this algebra and the associated loop algebra. 

\begin{theorem}\label{th:loop_universal_groups}
Let  $\pi:G\rightarrow \overline{G}$ be a surjective group homomorphism of abelian groups. Let $(\cA,\overline{\Gamma},\overline{G},\bar\delta)$ be a $\overline{G}$-graded algebra with universal group $(\overline{U},\delta^U_{\overline{\Gamma}})$, and let $\bar\alpha$ be the (unique) group homomorphism making commutative the diagram:
\[
\xymatrix{
&& \overline{U}\ar[dd]^{\bar\alpha}\\
\overline{\Gamma}\ar[rru]^-{\delta^U_{\overline{\Gamma}}}\ar[rrd]^-{\bar\delta} &&\\
&&\overline{G}
}
\]
(Recall that this implies that $\cA=\bigoplus_{\bar u\in\overline{U}}\cA'_{\bar u}$, with $\cA'_{\bar u}=\cA_{\bar\alpha({\bar u})}$ for any $\bar u\in\overline{U}$.)

Let $(L_\pi(\cA),\Gamma,G,\delta)$ be the associated loop algebra, with universal group $(U,\delta_\Gamma^U)$, and let $\alpha$ be the (unique) group homomorphism making commutative the diagram
\[
\xymatrix{
&& U\ar[dd]^{\alpha}\\
\Gamma\ar[rru]^{\delta_\Gamma^U}\ar[rrd]^{\delta} &&\\
&& G
}
\]
Then there exists a unique group homomorphism $\pi^U:U\rightarrow \overline{U}$ such that the diagram
\begin{equation}\label{eq:piU_pi_alphas_commute}
\begin{tabular}{c}
\xymatrix{
U\ar[r]^{\alpha}\ar[d]_{\pi^U} & G\ar[d]^{\pi}\\
\overline{U}\ar[r]^{\bar\alpha} & \overline{G}
}
\end{tabular}
\end{equation}
is commutative. Moreover, $\pi^U$ is surjective and the restriction $\alpha\vert_{\ker\pi^U}$ is a bijection $\ker\pi^U\rightarrow \ker\pi$.

 Besides, the $\overline{U}$-graded algebra $(\cA,\overline{\Gamma},\overline{U},\delta_{\overline{\Gamma}}^U)$ and the surjective group homomorphism $\pi^U:U\rightarrow \overline{U}$ define the loop algebra $(L_{\pi^U}(\cA),\Gamma^U,U,\delta_\Gamma^U)$, and the map
\[
\begin{split}
L_{\pi^U}(\cA)&\longrightarrow L_\pi(\cA)\\
x\otimes u\, &\mapsto\ x\otimes\alpha(u)
\end{split}
\]
for $x\in \cA'_{\pi^U(u)}$ and $u\in U$, is an equivalence $(L_{\pi^U}(\cA),\Gamma^U,U,\delta_\Gamma^U)\rightarrow (L_\pi(\cA),\Gamma,G,\delta)$.
\end{theorem}
\begin{proof}
The grading $\overline{\Gamma}$ is realized as a $\overline{G}$-grading: $\cA=\bigoplus_{\bar g\in\overline{G}}\cA_{\overline{g}}$ and as a $\overline{U}$-grading: $\cA=\bigoplus_{\bar u\in\overline{U}}\cA'_{\bar u}$, where $\cA'_{\bar u}=\cA_{\bar\alpha({\bar u})}$. Similarly for $\Gamma$: $L_\pi(\cA)=\bigoplus_{g\in G}L_{\pi}(\cA)_g$ and $L_\pi(\cA)=\bigoplus_{u\in U}L_{\pi}(\cA)'_u$, where 
\[
L_{\pi}(\cA)'_u=L_{\pi}(\cA)_{\alpha(u)}=\cA_{\pi\alpha(u)}\otimes \alpha(u).
\] 
For any $u\in \supp_U(\Gamma)$, there is a unique $\bar u\in \supp_{\overline{U}}(\overline{\Gamma})$ such that  $\cA_{\pi\alpha(u)}=\cA'_{\bar u}=\cA_{\bar\alpha(\bar u)}$, and this defines a unique group homomorphism $\pi^U:U\rightarrow \overline{U}$, $u\mapsto \bar u$ (note that $\supp_U(\Gamma)$ generates the universal group $U$ and the same for $\overline{\Gamma}$ and $\overline{U}$), such that the diagram \eqref{eq:piU_pi_alphas_commute} commutes. This homomorphism $\pi^U$ is surjective.

For any $h\in\ker\pi$ and $g\in\supp_G\Gamma$, $L_{\pi}(\cA)_g=\cA_{\pi(g)}\otimes g\neq 0$, and hence we have $L_{\pi}(\cA)_{gh}=\cA_{\pi(gh)}\otimes gh=\cA_{\pi(g)}\otimes gh\neq 0$, so that $gh\in\supp_G(\Gamma)$. Therefore $g=\alpha(u)$ and $gh=\alpha(u')$ for some $u,u'\in U$, and this shows that $h=\alpha(u'u^{-1})$. Therefore the restriction $\alpha\vert_{\ker \pi^U}:\ker\pi^U\rightarrow \ker\pi$ is onto. On the other hand, if $\tilde u\in \ker\pi^U\cap \ker\alpha$, and $g\in \supp_G(\Gamma)$, there is an element $u\in U$ with $\alpha(u)=g$, and
\[
L_{\pi}(\cA)'_{u\tilde u}=L_{\pi}(\cA)_{\alpha(u\tilde u)}=L_{\pi}(\cA)_{\alpha(u)}=\cA_{\pi(g)}\otimes g=L_\pi(\cA)'_{u},
\]
so $u\tilde u=u$ and $\tilde u=e$. This shows that the restriction $\alpha\vert_{\ker \pi^U}:\ker\pi^U\rightarrow \ker\pi$ is bijective.

The last part is clear.
\end{proof}

\begin{corollary}\label{co:loop_universal}
Let  $\pi:G\rightarrow \overline{G}$ be a surjective group homomorphism of abelian groups. Let $(\cA,\overline{\Gamma},\overline{G},\bar\delta)$ be a $\overline{G}$-graded algebra, and let $(L_{\pi}(\cA),\Gamma,G,\delta)$ be the associated loop algebra. Then $(\overline{G},\bar\delta)$ is, up to isomorphism, the universal group of $\overline{\Gamma}$ if and only if $(G,\delta)$ is, up to isomorphism, the universal group of $\Gamma$.
\end{corollary}

Recall that given an $\FF$-algebra $\cA$, its \emph{centroid} $C(\cA)$ is the centralizer in $\End_\FF(\cA)$ of the (associative) subalgebra $\Mult(\cA)$ generated by the left and right multiplications by elements in $\cA$. The algebra $\cA$ is said to be \emph{central simple} if it is simple and $C(\cA)=\FF 1$ (here $1$ denotes the unity of $\End_\FF(\cA)$, that is, the identity map on $\cA$). Central simple algebras are those simple algebras that remain simple after extension of scalars. (See \cite[Chapter X]{Jacobson} or \cite[Theorem I.2.5]{EldMyung}.)

Given an abelian group $G$ and a $G$-graded algebra $(\cA,\Gamma,G,\delta)$, we have $\cA=\bigoplus_{g\in G}\cA_g$, and we may consider the subspaces of the centroid: 
\[
C(\cA)_g=\{ c\in C(\cA)\mid c\cA_{g'}\subseteq \cA_{gg'}\ \forall g'\in G\}
\]
for any $g\in G$. In general $\bigoplus_{g\in G}C(\cA)_g$ may fail to be the whole centroid $C(\cA)$.

The $G$-graded algebra $(\cA,\Gamma,G,\delta)$ is said to be \emph{graded-central} if $C(\cA)_e=\FF1$, and \emph{graded-central-simple} if it is graded-central and graded-simple.

We collect in the next result the properties on loop algebras that will be needed later on:

\begin{theorem}\label{th:loop_properties}
Let $G$ be an abelian group.
\begin{enumerate}
\item If a $G$-graded algebra $(\cA,\Gamma,G,\delta)$ is graded-simple, then $C(\cA)=\bigoplus_{g\in G}C(\cA)_g$ (i.e., the centroid is $G$-graded too). This $G$-grading on $C(\cA)$ will be denoted by $(\Gamma_{C(\cA)},G,\delta_{C(\cA)})$.

\item If a $G$-graded algebra $(\cA,\Gamma,G,\delta)$ is graded-simple and $H$ denotes the support of the induced grading on the centroid: $H=\supp_G(\Gamma_{C(\cA)})$, then $H$ is a subgroup of $G$. Moreover, if the ground field $\FF$ is algebraically closed and $(\cA,\Gamma,G,\delta)$ is graded-central-simple, then the centroid $C(\cA)$ is isomorphic, as a $G$-graded algebra, to the group algebra $\FF H$ with its natural $G$-grading ($\bigl(\FF H\bigr)_h=\FF h$ for $h\in H$, and $\bigl(\FF H\bigr)_g=0$ for $g\in G\setminus H$).

\item Let $\pi:G\rightarrow \overline{G}$ be a surjective group homomorphism with kernel $H$ and let $(\cA,\overline{\Gamma},\overline{G},\bar\delta)$ be a central simple $\overline{G}$-graded algebra. Then the associated loop algebra $(L_\pi(\cA),\Gamma,G,\delta)$ is graded-central-simple and the map
\[
\begin{split}
\FF H&\longrightarrow C\bigl(L_\pi(\cA)\bigr)\\
 h\ &\mapsto\ \Bigl(x\otimes g\mapsto x\otimes hg\Bigr)
\end{split}
\]
for any $g\in G$ and $x\in \cA_{\pi(g)}$, is an isomorphism of $G$-graded algebras.

\item Let $(\cB,\tilde\Gamma,G,\tilde\delta)$ be a graded-central-simple $G$-graded algebra. Assume that the ground field $\FF$ is algebraically closed, and let $H$ be the support of the induced grading on the centroid: $H=\supp_G(\Gamma_{C(\cB)})$. Let $\pi:G\rightarrow \overline{G}$ be a surjective group homomorphism with kernel $H$. Then there exists a central simple $\overline{G}$-graded algebra $(\cA,\overline{\Gamma},\overline{G},\bar\delta)$ such that $(\cB,\tilde\Gamma,G,\tilde\delta)$ is isomorphic, as a $G$-graded algebra, to the associated loop algebra $(L_\pi(\cA),\Gamma,G,\delta)$. Moreover, the algebra $\cA$ is a quotient of $\cB$ (i.e., there is a surjective homomorphism of algebras $\cB\rightarrow \cA$).

\item Assume that the ground field $\FF$ is algebraically closed. Let $H^1$ and $H^2$ be subgroups of $G$, consider the quotient groups $\overline{G}^i=G/H^i$ and the natural projections $\pi^i:G\rightarrow \overline{G}^i$, $i=1,2$. Let $(\cA^i,\overline{\Gamma}^i,\overline{G}^i,\bar\delta^i)$ be a central simple $\overline{G}^i$-graded algebra for $i=1,2$. Then the associated loop algebras $(L_{\pi^1}(\cA^1),\Gamma^1,G,\delta^1)$ and $(L_{\pi^2}(\cA^2),\Gamma^2,G,\delta^2)$ are isomorphic, as $G$-graded algebras, if and only if $H^1=H^2$ and the $\overline{G}=G/H^1$-graded algebras $(\cA^1,\overline{\Gamma}^1,\overline{G}^1,\bar\delta^1)$ and $(\cA^2,\overline{\Gamma}^2,\overline{G}^2,\bar\delta^2)$ are isomorphic.
\end{enumerate}
\end{theorem}
\begin{proof}
The first part follows from \cite[Lemma 4.2.3]{loop}, the second from \cite[Lemma 4.2.3]{loop} and \cite[Lemma 4.3.8]{loop}, the third part from \cite[Lemma 5.1.3 and Proposition 5.2.3]{loop} and the fourth from \cite[Theorem 7.1.1(ii)]{loop} taking into account the second item. Finally, the last item follows from \cite[Remark 6.3.7 and Theorem 7.1.1(iii)]{loop}.
\end{proof}

The condition of a group-grading being fine is well behaved with respect to the loop algebra construction:

\begin{proposition}\label{pr:loop_fine}
Let $\pi:G\rightarrow \overline{G}$ be a surjective group homomorphism with kernel $H$, let $(\cA,\overline{\Gamma},\overline{G},\bar\delta)$ be a $\overline{G}$-graded algebra and let  $(L_\pi(\cA),\Gamma,G,\delta)$ be the associated loop algebra.
\begin{itemize}
\item If $\Gamma$ is a fine group-grading, so is $\overline{\Gamma}$.
\item If $\cA$ is central simple and $\overline{\Gamma}$ is a fine group-grading, so is $\Gamma$.
\end{itemize}
\end{proposition}
\begin{proof}
If $\Gamma$ is fine, as a group-grading, and $\overline{\Gamma'}$ is group-grading that properly refines $\overline{\Gamma}$ with universal group $(K,\bar\delta')$: $\overline{\Gamma'}:\cA=\bigoplus_{k\in K}\cA'_k$, then by Theorem \ref{th:universal} there is a group homomorphism $\varphi:K\rightarrow \overline{G}$ with $\varphi\bar\delta'=\bar\delta$. Then we may define a $G\times K$-grading on $L_\pi(\cA)$ by means of:
\[
L_\pi(\cA)_{(g,k)}=\begin{cases} \cA'_k\otimes g,&\text{if $\varphi(k)=\pi(g)$},\\
   0,&\text{otherwise,}
   \end{cases}
\]
that properly refines $\Gamma$, a contradiction.

Conversely, assuming that $\cA$ is central simple and $\overline{\Gamma}$ is a fine group-grading, let $\Gamma'$ be a group-grading that properly refines $\Gamma$ with universal group $(K,\delta')$. As before, there is a group homomorphism $\varphi:K\rightarrow G$ such that $\varphi\delta'=\delta$. For each $k\in K$, the homogeneous component $L_\pi(\cA)'_k$ is of the form $\cA'_k\otimes \varphi(k)$ for a subspace $\cA'_k$ contained in $\cA_{\pi\varphi(k)}$. Now, $(\Gamma',K,\delta')$ induces a $K$-grading $\Gamma'_C$ on the centroid $C\bigl(L_\pi(\cA)\bigr)$ that refines the one induced by $\Gamma$. By \ref{th:loop_properties}.(3), this last grading is fine with one-dimensional homogeneous components. It follows that the support $H'$ of $\Gamma'_C$ is a subgroup of $K$ and that $\varphi\vert_{H'}$ gives an isomorphism $H'\rightarrow H$. Hence $\varphi$ induces a homomorphism $\bar\varphi:K/H'\rightarrow G/H=\overline{G}$ and we may define a $K/H'$-grading on $\cA$ by means of $\cA'_{kH'}\bydef \cA_k$. This is well defined and provides a proper refinement of $\overline{\Gamma}$, a contradiction.
\end{proof}

We finish this section with the characterization of the semisimple loop algebras. Recall that, in this paper, a semisimple algebra is a finite direct sum of simple ideals or, alternatively, an algebra that is isomorphic to a finite cartesian product of simple algebras. If an $\FF$-algebra $\cA$ becomes semisimple after extending scalars to an algebraic closure $\overline{\FF}$, then it is semisimple. This is due to the fact that being semisimple means that it is completely reducible as a module for the multiplication algebra $\Mult(\cA)$, with only finitely many irreducible submodules, and if a module is completely reducible after extension of scalars, it is indeed completely reducible (see, for instance, \cite[Lemma III.4]{Jacobson}).

\begin{theorem}\label{th:loop_semisimple}
Let $\pi:G\rightarrow \overline{G}$ be a surjective group homomorphism of abelian groups with finite kernel $H=\ker\pi$. Let $(\cA,\overline{\Gamma},\overline{G},\bar\delta)$ be a central simple $\overline{G}$-graded algebra and let  $(L_\pi(\cA),\Gamma,G,\delta)$ be the associated loop algebra. Then $L_\pi(\cA)$ is semisimple if and only if the characteristic of $\FF$ does not divide the order of $H$.

If this is the case, and if the ground field $\FF$ is algebraically closed, then $L_\pi(\cA)$ is isomorphic to the cartesian product of $\lvert H\rvert$ copies of $\cA$.
\end{theorem}
\begin{proof}
If $\chr(\FF)$ divides $\lvert H\rvert$ then $c=\sum_{h\in H}h$ is a nonzero element of the group algebra $\FF H$ such that $c^2=\lvert H\rvert c=0$. By Theorem \ref{th:loop_properties}.(3) there is a nonzero element $\tilde c\in C\bigl(L_{\pi}(\cA)\bigr)$ with $\tilde c^2=0$ and hence the square of the nonzero ideal $\tilde cL_{\pi}(\cA)$ is zero, and $L_\pi(\cA)$ is not semisimple.

On the other hand, if $\chr(\FF)$ does not divide $n=\lvert H\rvert$, in order to show that $L_\pi(\cA)$ is semisimple, we may assume that $\FF$ is algebraically closed. In this case, the group of characters of $H$ consists of $n$ elements (see \cite[\S 5.6]{JacBAII}): $\widehat H=\{\chi_1,\ldots,\chi_n\}$. Also, as $\FF$ is algebraically closed, $\FF^\times$ is a divisible group, that is, an injective $\ZZ$-module (\cite[\S 3.11]{JacBAII}) and hence these characters may be extended to characters on the whole $G$. Consider the linear map:
\begin{equation}\label{eq:LpiA_An}
\begin{split}
\Phi: L_\pi(\cA)&\longrightarrow \cA\times\cdots\times\cA\quad \text{($n$ copies)}\\
x_{\pi(g)}\otimes g&\mapsto \bigl(\chi_1(g)x_{\pi(g)},\ldots,\chi_n(g)x_{\pi(g)}\bigr).
\end{split}
\end{equation}
The linear map $\Phi$ is a homomorphism of $\overline{G}$-graded algebras, where the $\overline{G}$-grading on $L_\pi(\cA)$ is given by the coarsening of $\Gamma$ induced by $\pi$, and the $\overline{G}$-grading on $\cA\times\cdots\times\cA$ is the product $\overline{G}$-grading (Definition \ref{df:G_product}). For any $\bar g\in\overline{G}$ in the support of $\overline{\Gamma}$, fix a pre-image $g\in G$, and let $H=\{h_1,\ldots,h_n\}$. The restriction of $\Phi$ to the homogeneous component of degree $\bar g$ is given by:
\[
\begin{split}
L_\pi(\cA)_{\bar g}=\bigoplus_{i=1}^n L_\pi(\cA)_{gh_i}=\bigoplus_{i=1}^n\cA_{\bar g}\otimes gh_i&\longrightarrow \cA_{\bar g}\times \cdots\times \cA_{\bar g}\\
\sum_{i=1}^n a_i\otimes gh_i\ &\mapsto\sum_{i=1}^n\bigl(\chi_1(gh_i)a_i,\ldots,\chi_n(gh_i)a_i\bigr)\\
&\qquad =(a_1,\ldots,a_n)\Bigl(\chi_j(gh_i)\Bigr)_{1\leq i,j\leq n}
\end{split}
\]
where $a_1,\ldots,a_n\in\cA_{\bar g}$. 
The linear independence of characters shows that 
the matrix $\Bigl(\chi_j(h_i)\Bigr)_{1\leq i,j\leq n}$ is regular, and thus so is $\Bigl(\chi_j(gh_i)\Bigr)_{1\leq i,j\leq n}$. Hence $\Phi$ is bijective on each nonzero homogeneous component of the $\overline{G}$-gradings.
\end{proof}

\begin{remark}\label{re:loop_semisimple}
The proof of Theorem \ref{th:loop_semisimple} shows that for a loop algebra $L_\pi(\cA)$ of a central simple algebra $\cA$, with the kernel of $\pi$ being finite, $L_\pi(\cA)$ is a direct sum of simple algebras if and only if it contains no proper ideal with zero square. The two natural notions of semisimplicity  agree in this case.
\end{remark}

\smallskip

\section{Classification up to isomorphism}\label{se:up_to_isomorphism}

Theorems \ref{th:loop_properties} and \ref{th:loop_semisimple} allow us, given an abelian group $G$, to classify $G$-gradings, up to isomorphism, in finite-dimensional semisimple algebras over an algebraically closed field, by reducing the problem to the analogous problem for simple algebras.

\begin{theorem}\label{th:up_to_isomorphism}
Let $\FF$ be an algebraically closed ground field, and let $G$ be an abelian group.
\begin{enumerate}
\item
Let $(\cB,\Gamma,G,\delta)$ be a semisimple $G$-graded algebra, then $(\cB,\Gamma,G,\delta)$ is isomorphic, as a $G$-graded algebra, to a product $G$-grading $\bigl(\cB^1\times\cdots\times\cB^n,\Gamma^1\times_G\cdots\times_G\Gamma^n,G,\delta^1\times_G\cdots\times_G\delta^n\bigr)$ for some graded-simple and semisimple $G$-graded algebras $(\cB^i,\Gamma^i,G,\delta^i)$, $i=1,\ldots,n$. The factors $(\cB^i,\Gamma^i,G,\delta^i)$ are uniquely determined up to reordering and $G$-graded isomorphisms.

\item 
Any finite-dimensional graded-simple $G$-graded algebra $(\cB,\tilde\Gamma,G,\tilde\delta)$ is isomorphic, as a $G$-graded algebra, to the loop algebra  $(L_\pi(\cA),\Gamma,G,\delta)$ associated to a surjective group homomorphism $\pi:G\rightarrow \overline{G}$ with finite kernel $H$, and  a central simple $\overline{G}$-graded algebra $(\cA,\overline{\Gamma},\overline{G},\bar\delta)$. The algebra $\cA$ is a quotient of $\cB$.

Moreover, in this situation $\cB$ is semisimple if and only if $\chr\FF$ does not divide the order of $H$.

\end{enumerate}
\end{theorem}
\begin{proof}
Let $(\cB,\Gamma,G,\delta)$ be a semisimple $G$-graded algebra. Then $\cB$ is a finite direct sum of simple ideals: $\cB=\bigoplus_{i=1}^m\cJ^i$. Thus $\cB$ contains only a finite number of ideals, as each ideal is a sum of some of the $\cJ^i$'s: If $\cJ$ is a nonzero ideal of $\cB$, any element $x\in\cJ$ can be written uniquely as $x=x_1+\cdots +x_m$ with $x_i\in\cJ^i$, $i=1,\ldots,m$. If $x_i\neq 0$, then by simplicity of $\cJ^i$, $0\neq x_i\cJ^i+\cJ^ix_i=x\cJ^i+\cJ^ix\subseteq \cJ\cap\cJ^i$, so $\cJ\cap\cJ^i$ is a nonzero ideal of the simple algebra $\cJ^i$, and hence $\cJ^i\subseteq \cJ$. If $(\cB,\Gamma,G,\delta)$ is not graded-simple, then it contains a proper graded-simple ideal $\cJ=\cJ^{i_1}\oplus\cdots\oplus\cJ^{i_r}$, for some $1\leq i_1<\cdots<i_r\leq m$. But then 
\[
\cJ'\bydef\bigoplus_{i\not\in\{i_1,\ldots,i_r\}}\cJ^i\ =\{x\in\cB: x\cJ=0=\cJ x\}
\]
is a graded ideal too, and $\cB=\cJ\oplus\cJ'$. If the semisimple $G$-graded ideal $\cJ'$ is not graded simple we apply the same argument to it. It follows that $\cB$ is the direct sum of its graded-simple ideals, each of which is a direct sum of some of its simple ideals. This proves the first part. (Note that this part does not require the ground field to be algebraically closed or the algebra to be finite-dimensional.)

Now, if $(\cB,\Gamma,G,\delta)$ is a finite-dimensional graded-simple and semisimple $G$-graded algebra, the neutral component of its centroid $C(\cB)_e$ is a finite field extension of $\FF$, and since this is algebraically closed, $C(\cB)_e=\FF 1$ and $(\cB,\Gamma,G,\delta)$ is graded-central-simple. Part (4) of Theorem \ref{th:loop_properties} and Theorem \ref{th:loop_semisimple} prove part (2).
\end{proof}

 Theorem \ref{th:loop_properties}.(5) gives the conditions for isomorphism of the graded-simple $G$-algebras in Theorem \ref{th:up_to_isomorphism}.(2).

\begin{example}\label{ex:sl2xsl2_iso}
Let $\FF$ be an algebraically closed field of characteristic not $2$, and let $G$ be an abelian group.
\cite[Theorem 3.49]{EKmon} shows that any $G$-grading on $\frsl_2$ is isomorphic to one of the following gradings:
\begin{itemize}
\item $\Gamma_{\frsl_2}^1(G,g)$ for an element $g\in G$, determined by $\deg(E)=g$, $\deg(H)=e$ and $\deg(F)=g^{-1}$. (If $g=e$, this is the trivial grading.)
\item $\Gamma_{\frsl_2}^2(G,T)$ for a subgroup $T$ of $G$ isomorphic to $\left(\ZZ/2\right)^2$, determined by $\deg(H)=a$, $\deg(E)=b$, $\deg(F)=c$, where $a,b,c$ are the nontrivial elements of $T$.
\end{itemize}
Moreover,  gradings of different types are not isomorphic and 
\begin{itemize}
\item $\Gamma_{\frsl_2}^1(G,g)$ is isomorphic to $\Gamma_{\frsl_2}^1(G,g')$ if and only if $g'\in\{g,g^{-1}\}$.
\item $\Gamma_{\frsl_2}^2(G,T)$ is isomorphic to $\Gamma_{\frsl_2}^2(G,T')$ if and only if $T=T'$.
\end{itemize}

Theorems \ref{th:up_to_isomorphism} and \ref{th:loop_semisimple} show that given any abelian group $G$, any $G$-grading on $\cL=\frsl_2\times\frsl_2$ making it a graded-simple algebra (i.e., the two copies of $\frsl_2$ are not graded ideals) is isomorphic to the grading of a loop algebra $(L_\pi(\frsl_2),\Gamma,G,\delta)$, where $\pi:G\rightarrow \overline{G}$ is a surjective group homomorphism with $\ker\pi$ of order $2$: $\ker\pi=\langle h\rangle$, $h$ of order $2$, obtained from a grading  $(\frsl_2,\overline{G},\overline{\Gamma},\bar\delta)$. The loop algebra is isomorphic to $\cL$ by means of the isomorphism in Equation \eqref{eq:LpiA_An}, which allows us to transfer easily the grading on $L_\pi(\frsl_2)$ to $\cL$.

If $\overline{\Gamma}$ is isomorphic to $\Gamma_{\frsl_2}^1(\overline{G},\bar g)$, for some $\bar g\in\overline{G}$, the corresponding grading on $\cL$ will be denoted by $\Gamma_{\cL}^1(G,h,\bar g)$, while if $\overline{\Gamma}$ is isomorphic to $\Gamma_{\frsl_2}^2(\overline{G},\overline{T})$ for a subgroup $\overline{T}$ of $\overline{G}$ isomorphic to $\left(\ZZ/2\right)^2$, the corresponding grading on $\cL$ will be denoted by $\Gamma_{\cL}^2(G,h,\overline{T})$. Then Theorem \ref{th:loop_properties}.(5) shows  that 
\begin{itemize}
\item $\Gamma_{\cL}^1(G,h,\bar g)$ is isomorphic to $\Gamma_{\cL}^1(G,h',\bar g')$ if and only if $h=h'$ and $\bar g'\in\{\bar g,\bar g^{-1}\}$.
\item $\Gamma_{\cL}^2(G,h,\overline{T})$ is isomorphic to $\Gamma_{\cL}^2(G,h',\overline{T}')$ if and only if $h=h'$ and $\overline{T}=\overline{T}'$.
\item A grading $\Gamma_{\cL}^1(G,h,\bar g)$ is never isomorphic to a grading $\Gamma_{\cL}^2(G,h,\overline{T})$.
\end{itemize}

The gradings $\Gamma_{\cL}^1(G,h,\bar g)$ are quite simple to describe if the surjective group homomorphism $\pi:\pi^{-1}(\langle\bar g\rangle)\rightarrow \langle \bar g\rangle$ splits. That is, if there is an element $g\in G$ with $\pi(g)=\bar g$ and $h$ does not belong to the subgroup generated by $g$. In this case $\pi^{-1}(\langle \bar g\rangle)$ is the direct product of the subgroups $\langle g\rangle$ and $\langle h\rangle$. In this situation, the nontrivial character on $\ker\pi=\langle h\rangle$ ($\chi(h)=-1$) extends to a  character $\chi$ on $G$ with $\chi(g)=1$. The isomorphism in \eqref{eq:LpiA_An} becomes here the isomorphism
\[
\begin{split}
\Phi:L_\pi(\frsl_2)&\longrightarrow \cL=\frsl_2\times\frsl_2\\
      x\otimes f&\mapsto \ \bigl(x,\chi(f)x\bigr)
\end{split}
\]
for $f\in G$ and  $x\in \left(\frsl_2\right)_{\pi(f)}$. Thus the $G$-grading $\Gamma_{\cL}^1(G,h,\bar g)$ is determined by:
\begin{equation}\label{eq:Gamma_LGhg}
\begin{array}{ll}
\deg(H,H)=e,\qquad& \deg(H,-H)=h,\\
 \deg(E,E)=g,& \deg(E,-E)=gh,\\
  \deg(F,F)=g^{-1},& \deg(F,-F)=g^{-1}h.
\end{array}
\end{equation}

In the same vein, the gradings $\Gamma_{\cL}^2(G,h,\overline{T})$ are very easy to describe in case $\pi^{-1}(\overline{T})$ is a $2$-elementary subgroup (of order $8$). That is, if there are order $2$ elements $a,b\in G$, such that $\pi^{-1}(\overline{T})=\langle h,a,b\rangle$. In this situation, the character $\chi$ can be taken to satisfy $\chi(a)=\chi(b)=1$, and the $G$-grading $\Gamma_{\cL}^2(G,h,\overline{T})$ is given by:
\begin{equation}\label{eq:Gamma_L2GhT}
\begin{array}{ll}
\cL_a= \FF(H,H), & \cL_{ah}= \FF(H,-H),\\
\cL_b= \FF(E+F,E+F), & \cL_{bh}= \FF(E+F,-(E+F)),\\
\cL_{ab}= \FF(E-F,E-F), & \cL_{abh}= \FF(E-F,-(E-F)).
\end{array}
\end{equation}

In Example \ref{ex:sl2.cartesian.product} there appears the case in which $\pi^{-1}(\overline{T})$ is not $2$-elementary abelian, and hence it is isomorphic to $\ZZ/4\times\ZZ/2$.
\end{example}

\smallskip

\section{Fine group-gradings on semisimple algebras}\label{se:fine}

Fine (general) gradings behave very well with respect to cartesian products:

\begin{proposition}\label{pr:fine_general}
Let $\Gamma$ be a fine grading on an algebra which is a direct sum of graded ideals $\cB=\cB^1\oplus\cdots\oplus\cB^n$. Let $\Gamma^i=\Gamma\vert_{\cB^i}$ be the induced grading on $\cB^i$ for each $i=1,\ldots,n$.  Then each $\Gamma^i$ is a fine grading, and $\Gamma$ is equivalent to the product grading $\Gamma\vert_{\cB^1}\times\cdots\times\Gamma\vert_{\cB^n}$ on $\cB^1\times\cdots\times\cB^n$ (naturally isomorphic to $\cB$).

Conversely, let $\Gamma^i$ be a fine grading on $\cB^i$, for $i=1,\ldots,n$, then the product grading $\Gamma^1\times\cdots\times\Gamma^n$ is a fine grading on $\cB^1\times\cdots\times\cB^n$.
\end{proposition}
\begin{proof}
The grading $\tilde\Gamma=\cup_{i=1}^n\Gamma\vert_{\cB^i}=\cup_{i=1}^n\Gamma^i$ is a refinement of $\Gamma$ so, as $\Gamma$ is fine, they are equal. Under the natural isomorphism $\cB\cong \cB^1\times\cdots\times\cB^n$, $\tilde\Gamma$ becomes the product grading $\Gamma^1\times\cdots\times\Gamma^n$. Besides, any refinement of $\Gamma^i$ for an index $i$ gives a refinement of the union $\Gamma$, and hence each $\Gamma^i$ is fine. 

The converse is clear.
\end{proof}

However, for fine group-gradings, only one direction works:

\begin{proposition}\label{pr:fine_group}
Let $\Gamma$ be a fine group-grading on an algebra which is a direct sum of graded ideals: $\cB=\cB^1\oplus\cdots\oplus\cB^n$. Let $\Gamma^i=\Gamma\vert_{\cB^i}$ be the induced grading on $\cB^i$ for each $i=1,\ldots,n$.
Then each $\Gamma^i$ is a fine group-grading, and $\Gamma$ is equivalent to the free product group-grading $\bigl(\Gamma^1\times\cdots\times\Gamma^n\bigr)_{\mathrm{gr}}$ on $\cB^1\times\cdots\times\cB^n$ (naturally isomorphic to $\cB$).
\end{proposition}
\begin{proof} 
Through the natural isomorphism $\cB\cong\cB^1\times\cdots\times\cB^n$, $\Gamma$ becomes a grading on $\cB^1 \times\cdots\times\cB^n$ which is a coarsening of the free product group-grading $\bigl(\Gamma^1\times\cdots\times\Gamma^n\bigr)_{\mathrm{gr}}$  by Theorem \ref{th:grading_group-grading}. Since $\Gamma$ is fine, they are equivalent. Again, any group-grading that properly refines any of the $\Gamma^i$, gives a proper refinement of $\bigl(\Gamma^1\times\cdots\times\Gamma^n\bigr)_{\mathrm{gr}}$. Since this is a fine group-grading, it follows that each $\Gamma^i$ is a fine group-grading.
\end{proof}

\begin{example}\label{ex:FxFtri}
Assume that the characteristic of $\FF$ is not $2$. The trivial group-grading on $\FF\times\FF$ is the product group-grading of the trivial group-grading on each copy of $\FF$. Observe that the trivial grading on $\FF$ is fine. However,  as observed in Remark \ref{rk:fine_group_grading}, the grading by  the cyclic group of order $2$: $C_2=\{e,g\}$, with $(\FF\times\FF)_e=\FF(1,1)$ and $(\FF\times\FF)_g=\FF(1,-1)$, shows that the trivial grading on $\FF\times\FF$ is not fine. Hence, in general, the free product group-grading of fine group-gradings is not necessarily a fine group-grading.
\end{example}

There appears a natural question:

\begin{center}
\emph{Under what conditions the free product group-grading of fine group-gradings is a fine group-grading?}
\end{center}

For nontrivial fine group-gradings on graded-simple algebras (and this is the situation for finite-dimensional semisimple Lie algebras) the answer is easy:

\begin{theorem}\label{th:fine_nontrivial}
For $i=1,\ldots,n$, let $\Gamma^i$ be a fine group-grading on an algebra $\cB^i$ such that $(\cB^i,\Gamma^i)$ is graded-simple and $\Gamma^i$ is not trivial. Then the free product group-grading $(\Gamma^1\times\cdots\times\Gamma^n)_{\mathrm{gr}}$ on $\cB^1\times\cdots\times\cB^n$ is a fine group-grading.
\end{theorem}
\begin{proof}
Identify $\cB^i$ with the ideal $0\times\cdots\times\cB^i\times\cdots\times 0$ of $\cB=\cB^1\times\cdots\times\cB^n$, so we may write $\cB=\cB^1\oplus\cdots\oplus\cB^n$. Let $(U^i,\delta^i)$ be the universal group of $\Gamma^i$, $i=1,\ldots,n$, and assume that there is a group-grading $\Gamma'$ on $\cB$ that refines $(\Gamma^1\times\cdots\times\Gamma^n)_{\mathrm{gr}}$.

As $\Gamma^i$ is not trivial, there is an element $e\neq u\in U^i$ such that $(\cB^i)_u\neq 0$. Then by Theorem \ref{th:product_group_grading}, the universal group of $(\Gamma^1\times\cdots\times\Gamma^n)_{\mathrm{gr}}$ is $U^1\times\cdots\times U^n$ and $(\cB^i)_u=\cB_{(e,\ldots,u,\ldots,e)}$. Since $\Gamma'$ refines $(\Gamma^1\times\cdots\times\Gamma^n)_{\mathrm{gr}}$, $(\cB^i)_u$ is a sum of homogeneous spaces of $\Gamma'$, so there is a nonzero element $x\in (\cB^i)_u$ which is homogeneous for $\Gamma'$. But by the graded-simplicity of $(\cB^i,\Gamma^i)$, the ideal of $\cB$ generated by the homogeneous element $x$ is $\cB^i$, and hence $\cB^i$ is a graded ideal for $\Gamma'$. Since $\Gamma^i$ is fine and $\Gamma'\vert_{\cB^i}$ refines it, we obtain $\Gamma'\vert_{\cB^i}=\Gamma^i$ for any $i$, and thus $\Gamma'=(\Gamma^1\times\cdots\times\Gamma^n)_{\mathrm{gr}}$  by Proposition \ref{pr:fine_group}.
\end{proof}

The next example shows that there are nontrivial examples of simple algebras for which the trivial grading is a fine group-grading. 

\begin{example}\label{ex:trivial_fine}
Let $\cA=\FF a\oplus\FF b$ be the algebra with $a^2=a$, $ab=b$, $ba=0$, and $b^2=a+b$. It is easy to see that $\cA$ is simple. For any $\cR$ in $\Alg_\FF$ and any automorphism $\varphi\in\Aut_\cR(\cA\otimes_\FF\cR)$, $\varphi(a)=a$ because $a$ is the only left unity of $\cA\otimes_\FF\cR$. If $\varphi(b)=ra+sb$, $r,s\in\cR$, then from $0=\varphi(ba)=\varphi(b)a=(ra+sb)a=ra$ we obtain $r=0$. Now $\varphi(b^2)=\varphi(a+b)=a+sb$, while $\varphi(b)^2=s^2(a+b)$, so $s^2=1=s$ and $\varphi$ is the identity. Therefore the affine group scheme $\AAut_\FF(\cA)$ is trivial, and hence the only group-grading is the trivial one.
\end{example}

The situation in Theorem \ref{th:fine_nontrivial} changes in presence of simple algebras admitting no nontrivial group-gradings, that is, simple algebras for which the trivial grading is fine. We first need a previous result.

\begin{lemma}\label{le:trivial}
Let $\cA$ be a central simple algebra over an algebraically closed field $\FF$ with no nontrivial group-gradings. Then the trivial grading on $\cA\times\cdots\times\cA$ ($n\geq 2$ copies of $\cA$) is a fine group-grading if and only if $n=2$ and $\chr\FF=2$.
\end{lemma}
\begin{proof}
If $n=2$ and $\chr\FF\neq 2$, then $\cA\times\cA$ is isomorphic to $\cA\otimes_\FF(\FF\times\FF)$. The
$C_2$-grading on $\FF\times\FF$ in Example \ref{ex:FxFtri} induces a nontrivial $C_2$-grading on 
$\cA\times\cA$, with $(\cA\times\cA)_e=\{(x,x)\mid x\in \cA\}$ and 
$(\cA\times\cA)_g=\{(x,-x)\mid x\in \cA\}$. Therefore the trivial grading 
on $\cA\times\cA$ is not fine.

If $n\geq 3$ and $\chr\FF\neq 2$, we may use the above to define a nontrivial grading on $\cA\times\cA$ and hence take the product grading with the trivial grading on the remaining factors to get a nontrivial grading on $\cA\times\cA\times\cdots\times\cA$. 

If $n\geq 3$ and $\chr\FF=2$, consider the cyclic group of order $3$: $C_3$, and its projection onto the trivial group $\pi:C_3\rightarrow 1$. Take $G=C_3$, $\overline{G}=1$ and $\overline{\Gamma}$ the trivial grading on $\cA$. Consider the associated loop algebra $(L_\pi(\cA),\Gamma,G,\delta)$. Its grading $\Gamma$ is not trivial, and $L_\pi(\cA)$ is isomorphic to $\cA\times\cA\times\cA$ (Theorem \ref{th:loop_semisimple}). Therefore there are nontrivial group-gradings on $\cA\times\cA\times\cA$, and hence also on the cartesian product of $n\geq 3$ copies of $\cA$.

On the other hand, if $n=2$ and $\chr\FF=2$, and if $\Gamma$ were a nontrivial group-grading on $\cA\times\cA$ with universal group $(U,\delta)$, then $(\cA\times\cA,\Gamma,U,\delta)$ would be a semisimple and graded-simple algebra, with centroid isomorphic to $\FF\times\FF$. By Theorem \ref{th:up_to_isomorphism}.(2) $(\cA\times\cA,\Gamma,U,\delta)$ would be isomorphic to a loop algebra of the form $(L_\pi(\cA'),\Gamma',U,\delta')$, with $\pi:U\rightarrow\overline{U}$ a surjective group homomorphism with kernel $H$ of order $2$ and a central simple $\overline{U}$-graded algebra $\cA'$. But then $L_\pi(\cA')$ would be semisimple, and this would contradict Theorem \ref{th:loop_semisimple}.
\end{proof}

\begin{theorem}\label{th:fine}
Let the ground field $\FF$ be algebraically closed.
For $i=1,\ldots,n$, let $\Gamma^i$ be a fine group-grading on an algebra $\cB^i$ such that $(\cB^i,\Gamma^i,U(\Gamma^i),\delta_{\Gamma^i}^U)$ is graded-central-simple.
Then the free product group-grading $(\Gamma^1\times\cdots\times\Gamma^n)_{\mathrm{gr}}$ on $\cB^1\times\cdots\times\cB^n$ is a fine group-grading if and only if either:
\begin{itemize}
\item 
$\chr\FF=2$ and for any index $i$ such that $\Gamma^i$ is trivial, there is at most one other index $j$ such that $(\cB^i,\Gamma^i)$ and $(\cB^j,\Gamma^j)$ are equivalent (i.e., $\Gamma^j$ is trivial and $\cB^j$ is isomorphic to $\cB^i$).

\item 
$\chr\FF\neq2$ and for any index $i$ such that $\Gamma^i$ is trivial, there is no other index $j$ such that $(\cB^i,\Gamma^i)$ and $(\cB^j,\Gamma^j)$ are equivalent.
\end{itemize}
\end{theorem}
\begin{proof}
As in the proof of Theorem \ref{th:fine_nontrivial}, identify $\cB^i$ with the ideal $0\times\cdots\times\cB^i\times\cdots\times 0$ of $\cB=\cB^1\times\cdots\times\cB^n$, and write $\cB=\cB^1\oplus\cdots\oplus\cB^n$.

If $\chr\FF=2$ and there are three indices $i,j,k$ with trivial gradings $\Gamma^i$, $\Gamma^j$ and $\Gamma^k$ and such that $\cB^i$, $\cB^j$ and $\cB^k$ are isomorphic, then the graded-central-simplicity implies the central simplicity of these algebras, and Lemma \ref{le:trivial} shows that there is a nontrivial grading on $\cB^i\oplus\cB^j\oplus\cB^k$. Therefore the induced grading
$(\Gamma^1\times\cdots\times\Gamma^n)_{\mathrm{gr}}\vert_{\cB^i\oplus\cB^j\oplus\cB^k}$ (which is the trivial grading) is not fine and Proposition \ref{pr:fine_group} shows that $(\Gamma^1\times\cdots\times \Gamma^n)_{\mathrm{gr}}$ is not fine. The situation for $\chr\FF\neq 2$, with two indices $i,j$ with trivial gradings $\Gamma^i$ and $\Gamma^j$ and such that $\cB^i$ and $\cB^j$ are isomorphic, is similar.

On the other hand, assume that the hypotheses on the trivial gradings are satisfied. The arguments in the proof of Theorem \ref{th:fine_nontrivial} show that if $\Gamma'$ is a group-grading on $\cB$ that refines $(\Gamma^1\times\cdots\times\Gamma^n)_{\mathrm{gr}}$, and if $\Gamma^i$ is not trivial, then $\cB^i$ is a graded ideal for $\Gamma'$, and $\Gamma'\vert_{\cB^i}$ coincides with $\Gamma^i$.

Consider the subset of indices
\[
J=\{i\mid 1\leq i\leq n\ \text{and $\Gamma^i$ is the trivial grading: $\cB^i=(\cB^i)_e$}\}.
\]
As $\cJ=\bigoplus_{i\not\in J}\cB^i$ is a graded ideal for $\Gamma'$, the first arguments in the proof of Theorem \ref{th:up_to_isomorphism} show that $\cJ'=\bigoplus_{j\in J}\cB^j$ is also a graded ideal of $\cB$ for $\Gamma'$. For $j\in J$, $\cB^j$ is central simple (being graded-central-simple and $\Gamma^j$ being trivial), so $\cJ'$ is semisimple. Thus each ideal of $\cJ'$ is of the form $\cB^{j_1}\oplus\cdots\oplus\cB^{j_r}$ for some indices $j_1<\cdots<j_r$ in $J$. The centroid of such an ideal is the cartesian product of $r$ copies of $\FF$. In particular its dimension is finite.

Then $\cJ'$ is a direct sum of graded-central-simple ideals for $\Gamma'\vert_{\cJ'}$ (Theorem \ref{th:up_to_isomorphism}.(1)), each of which is isomorphic to a loop algebra (Theorem \ref{th:loop_properties}.(4)) with finite kernel, and hence isomorphic to the cartesian product of a number of copies of a simple algebra, with $\chr\FF$ not dividing this number of copies (Theorem \ref{th:loop_semisimple}). Our hypotheses imply, due to Lemma \ref{le:trivial}, that the graded-simple ideals of $\Gamma'\vert_{\cJ'}$ are precisely the $\cB^j$'s, $j\in J$ and, since $\Gamma^j$ is fine, $\Gamma'\vert_{\cB^j}$ is the trivial grading for any $j\in J$. We conclude that $\Gamma'=(\Gamma^1\times\cdots\times\Gamma^n)_{\mathrm{gr}}$.
\end{proof}

Our last result classifies fine group-gradings, up to equivalence, in finite-dimensional semisimple algebras over an algebraically closed field. Recall that a finite-dimensional (graded-)simple algebra over an algebraically closed field is (graded-)central-simple. (See the last part of the proof of Theorem \ref{th:up_to_isomorphism}.)

\begin{corollary}\label{co:fine_semisimple}
Let the ground field $\FF$ be algebraically closed.
\begin{enumerate}
\item 
Any fine group-grading on a finite-dimensional semisimple algebra is equivalent to a free product group-grading $(\Gamma^1\times\cdots\times\Gamma^n)_{\mathrm{gr}}$, with the $\Gamma^i$'s being fine group-gradings on a semisimple graded-simple algebra $\cB^i$, satisfying one of the following extra conditions:
\begin{itemize}
\item 
$\chr\FF=2$ and for any index $i$ such that $\Gamma^i$ is trivial, there is at most one other index $j$ such that $(\cB^i,\Gamma^i)$ is equivalent to $(\cB^j,\Gamma^j)$.

\item 
$\chr\FF\neq2$ and for any index $i$ such that $\Gamma^i$ is trivial, there is no other index $j$ such that $(\cB^i,\Gamma^i)$ is equivalent to $(\cB^j,\Gamma^j)$.
\end{itemize}
And conversely, any such free product group-grading is a fine group-grading.

Moreover, the factors $(\cB^i,\Gamma^i)$ are uniquely determined, up to reordering and equivalence.

\item
Any finite-dimensional graded-simple algebra $(\cB,\Gamma')$ such that $\Gamma'$ is a fine group-grading is equivalent to a loop algebra $(L_\pi(\cA),\Gamma,U,\delta)$ associated to a surjective group homomorphism $\pi:U\rightarrow\overline{U}$ with finite kernel, and a simple finite-dimensional graded algebra $(\cA,\overline{\Gamma},\overline{U},\bar\delta)$ with $\overline{\Gamma}$ a fine group-grading with universal group $(\overline{U},\bar\delta)$.

 Conversely, if $(\cA,\overline{\Gamma},\overline{U},\bar\delta)$ is a simple finite-dimensional graded algebra with $\overline{\Gamma}$ a fine group-grading with universal group $(\overline{U},\bar\delta)$, and $\pi:U\rightarrow \overline{U}$ is a surjective group homomorphism with finite kernel, then the corresponding grading on the associated loop algebra is fine.

Moreover, in this situation $\cB$ is semisimple if and only if $\chr\FF$ does not divide the order of $\ker\pi$.

\item
For $i=1,2$, let $(\cA^i,\overline{\Gamma}^i,\overline{U}^i,\bar \delta^i)$ consist of a simple algebra $\cA^i$ endowed with a fine group-grading $\overline{\Gamma}^i$ with universal group $(\overline{U}^i,\bar \delta^i)$, and let $\pi^i: U^i\rightarrow \overline{U}^i$ be a surjective group homomorphism. Let $(L_{\pi^i}(\cA^i),\Gamma^i,U^i,\delta^i)$ be the associated loop algebra. Then the group-graded algebras $(L_{\pi^1}(\cA^1),\Gamma^1)$ and $(L_{\pi^2}(\cA^2),\Gamma^2)$ are equivalent if and only if the group-graded algebras $(\cA^1,\overline{\Gamma}^1)$ and $(\cA^2,\overline{\Gamma}^2)$ are equivalent and there is an equivalence $\varphi:(\cA^1,\overline{\Gamma}^1)\rightarrow (\cA^2,\overline{\Gamma}^2)$ such that the associated group isomorphism $\alpha^U_\varphi:\overline{U}^1\rightarrow \overline{U}^2$ in Proposition \ref{pr:equivalence} extends to a group isomorphism $\tilde\alpha^U_\varphi:U^1\rightarrow U^2$. (This means that the diagram
\[
\xymatrix{
U^1\ar[r]^{\tilde\alpha^U_\varphi}\ar[d]_{\pi^1} & U^2\ar[d]^{\pi^2}\\
\overline{U}^1\ar[r]^{\alpha^U_\varphi} & \overline{U}^2
}
\]
is commutative.)
\end{enumerate}
\end{corollary}
\begin{proof}
As in the proof of Theorem \ref{th:up_to_isomorphism}.(1) any semisimple graded algebra is uniquely, up to a permutation of the summands, a direct sum of graded-simple ideals. Now part (1) follows from Proposition \ref{pr:fine_group} and Theorem \ref{th:fine}.

Part (2) follows from Theorem \ref{th:up_to_isomorphism}.(2)   and Proposition \ref{pr:loop_fine}.

For part (3) note that by Corollary \ref{co:loop_universal}, $(U^i,\delta^i)$ is, up to isomorphism, the universal group of $(L_{\pi^i}(\cA^i),\Gamma^i)$ for $i=1,2$. An equivalence $\psi:(L_{\pi^1}(\cA^1),\Gamma^1)\rightarrow (L_{\pi^2}(\cA^2),\Gamma^2)$ induces, by Proposition \ref{pr:equivalence}, a group isomorphism $\alpha^U_{\psi}:U^1\rightarrow U^2$. Since $\ker\pi^i$ is the support of the grading induced by $\Gamma^i$ on the centroid $C\bigl(L_{\pi^i}(\cA^i)\bigr)$, it follows that $\alpha^U_\psi$ takes $\ker\pi^1$ to $\ker\pi^2$, and hence induces an isomorphism $\bar\alpha^U_\psi:\overline{U}^1\rightarrow \overline{U}^2$. By Proposition \ref{pr:equivalence_iso}, $\psi$ becomes an isomorphism of $U^2$-graded algebras, which induces an isomorphism $\bar\psi:\cA^1\rightarrow \cA^2$ of $\overline{U}^2$-graded algebras by Theorem \ref{th:loop_properties}.(5) which, in turn, is an equivalence $\bar\psi:(\cA^1,\overline{\Gamma}^1)\rightarrow(\cA^2,\overline{\Gamma}^2)$ whose associated isomorphism $\overline{U}^1\rightarrow\overline{U}^2$ is $\bar\alpha^U_\psi$.

The converse is clear, if $\varphi:(\cA^1,\overline{\Gamma}^1)\rightarrow(\cA^2,\overline{\Gamma}^2)$ is an equivalence whose associated group isomorphism $\alpha^U_\varphi:\overline{U}^1\rightarrow \overline{U}^2$ extends to a group isomorphism $\tilde\alpha^U_\varphi:{U}^1\rightarrow {U}^2$, then the map $\psi$ given by $x\otimes g\mapsto \varphi(x)\otimes\tilde\alpha^U_\varphi(g)$, for $g\in U$ and $x\in (\cA^1)_{\pi^1(g)}$, is an equivalence $(L_{\pi^1}(\cA^1),\Gamma^1)\rightarrow (L_{\pi^2}(\cA^2),\Gamma^2)$.
\end{proof}

In general, the group isomorphism $\alpha^U_\varphi$ at the end of the previous proof cannot be extended to a group isomorphism $U^1\rightarrow U^2$, as our next example shows.

\begin{example}\label{ex:no_extension}
Let $\cJ=\FF 1\oplus\FF u\oplus\FF v$ be the unital commutative algebra with $u^2=v^2=1$, $uv=0$. This is the Jordan algebra of a two-dimensional quadratic form. It is simple. 

Consider the group-grading $\overline{\Gamma}^1$ on $\cJ$ by $\overline{U}=(\ZZ/2)^2$, with 
\[
\cJ_{(\bar{0},\bar{0})}=\FF 1,\quad \cJ_{(\bar{1},\bar{0})}=\FF u,\quad \cJ_{(\bar{1},\bar{1})}=\FF v.
\]
 If $\bar\delta^1$ denotes the corresponding degree map, $(\overline{U},\bar\delta^1)$ is, up to isomorphism, the universal group of the grading. 
 
Consider also the group-grading $\overline{\Gamma}^2$ by the same group with 
\[
\cJ_{(\bar{0},\bar{0})}=\FF 1,\quad \cJ_{(\bar{0},\bar{1})}=\FF u,\quad \cJ_{(\bar{1},\bar{1})}=\FF v.
\]
Again, with $\bar\delta^2$ being the degree map of $\overline{\Gamma}^2$, $(\overline{U},\bar\delta^2)$ is the universal group of $\Gamma^2$.

The identity map gives an equivalence $\id:(\cJ,\overline{\Gamma}^1)\rightarrow(\cJ,\overline{\Gamma}^2)$. The associated group isomorphism $\alpha^U_\id:\overline{U}\rightarrow\overline{U}$ is the swap map $(a,b)\mapsto (b,a)$.

Let $U=\ZZ/4\times\ZZ/2$, and let $\pi$ be the natural projection map $U\rightarrow \overline{U}$ which is the identity on the second component and the projection $\ZZ/4\rightarrow \ZZ/2$ on the first component. Then $\alpha^U_\id$ does not extend to a group isomorphism $U\rightarrow U$, and therefore $\id$ does not extend to an equivalence of the associated loop algebras $(L_{\pi}(\cJ),\Gamma^1,U,\delta^1)$ and $(L_\pi(\cJ),\Gamma^2,U,\delta^2)$.  The same happens with any other equivalence $\varphi:(\cJ,\overline{\Gamma}^1)\rightarrow(\cJ,\overline{\Gamma}^2)$.

If $\FF$ is algebraically closed of characteristic not $2$, then $L_\pi(\cJ)$ is isomorphic to $\cJ\times\cJ$ (Theorem \ref{th:loop_semisimple}), and we obtain two non-equivalent gradings on the semisimple algebra $\cJ\times\cJ$.
\end{example}

We finish the paper going back to the example $\cL=\frsl_2\times\frsl_2$.

\begin{example}\label{ex:sl2.cartesian.product}
Let $\FF$ be an algebraically closed field of characteristic not $2$. Corollary \ref{co:fine_semisimple} tells us that, up to equivalence, the fine gradings on $\cL=\frsl_2\times\frsl_2$ are:
\begin{itemize}
\item The free product group-gradings $(\Gamma_{\frsl_2}^1\times \Gamma_{\frsl_2}^1)_{\mathrm{gr}}$, $(\Gamma_{\frsl_2}^1\times \Gamma_{\frsl_2}^2)_{\mathrm{gr}}$, and $(\Gamma_{\frsl_2}^2\times \Gamma_{\frsl_2}^2)_{\mathrm{gr}}$ in Example \ref{ex:sl2.product.group-grading}, with respective universal groups $\ZZ^2$, $\ZZ\times\bigl(\ZZ/2\bigr)^2$, and $\bigl(\ZZ/2\bigr)^4$.

\item The grading $\Gamma_\cL^1(\ZZ\times\ZZ/2,(0,\bar 1),1)$ in Example \ref{ex:sl2xsl2_iso}, with universal group $U=\ZZ\times\ZZ/2$. The group $\overline{U}=U/\langle (0,\bar 1)\rangle$ is identified naturally with $\ZZ$. Note that the fine group-grading $\Gamma_{\frsl_2}^1$ is precisely $\Gamma_{\frsl_2}^1(\ZZ,1)$. This grading is determined explicitly using Equation \eqref{eq:Gamma_LGhg}.

\item The grading $\Gamma_\cL^2\Bigl(\bigl(\ZZ/2\bigr)^3,(\bar 0,\bar 0,\bar 1),\bigl(\ZZ/2\bigr)^2\Bigr)$ in Example \ref{ex:sl2xsl2_iso}, with universal group $U=\bigl(\ZZ/2\bigr)^3$. Here the group $\overline{U}=U/\langle (\bar 0,\bar 0,\bar 1)\rangle$ is identified with $\bigl(\ZZ/2\bigr)^2$. The fine group-grading $\Gamma_{\frsl_2}^2$ is $\Gamma^2_{\frsl_2}(G,T)$ with $G=T=\bigl(\ZZ/2\bigr)^2$.
This grading is determined explicitly using Equation \eqref{eq:Gamma_L2GhT}.

\item The grading $\Gamma_\cL^2(\ZZ/4\times\ZZ/2,(\widehat{2},\bar 0),\ZZ/2\times\ZZ/2)$. Here we denote by $\widehat{m}$ the class of the integer $m$ modulo $4$ and restrict the usual notation $\bar{m}$ for the class of $m$ modulo $2$. The surjective group homomorphism $\pi$ is the canonical homomorphism $\ZZ/4\times\ZZ/2\rightarrow \ZZ/2\times\ZZ/2$, $(\widehat{m},\bar n)\mapsto (\bar{m},\bar n)$.

Let us give a precise description of this grading. The nontrivial character $\chi$ on $\langle h=(\widehat{2},\bar 0)\rangle$ extends to the character $\chi$ on $U=\ZZ/4\times\ZZ/2$ by $\chi(\widehat{m},\bar n)=\mathbf{i}^m$, where $\mathbf{i}$ denotes a square root of $-1$ in $\FF$. 

The grading on the loop algebra $L_\pi(\frsl_2)$ is given by
\[
L_{\pi}(\frsl_2)_{(\widehat{m},\bar n)}= (\frsl_2)_{(\bar{m},\bar n)} \otimes (\widehat{m},\bar n)
\]
for the homogeneous components $(\frsl_2)_{(\bar{m} ,\bar n)}$ in Equation \eqref{eq:Gamma2sl2}, and through the isomorphism $\Phi$ in Equation \eqref{eq:LpiA_An}, our grading $\Gamma_\cL^2(\ZZ/4\times\ZZ/2,(\widehat{2},\bar 0),\ZZ/2\times\ZZ/2)$ on $\cL=\frsl_2\times\frsl_2$ is given by
\[
\cL_{(\widehat{m},\bar n)}=\{(x,\mathbf{i}^mx)\mid x\in (\frsl_2)_{(\bar{m},\bar n)}\}.
\]
That is,
\[
\begin{array}{ll}
\cL_{(\widehat{1},\bar 0)}= \FF(H,\mathbf{i}H), & \cL_{(\widehat{3},\bar 0)}= \FF(H,-\mathbf{i}H),\\
\cL_{(\widehat{0},\bar 1)}= \FF(E+F,E+F), & \cL_{(\widehat{2},\bar 1)}= \FF(E+F,-(E+F)),\\
\cL_{(\widehat{1},\bar 1)}= \FF(E-F,\mathbf{i}(E-F)), & \cL_{(\widehat{3},\bar 1)}= \FF(E-F,-\mathbf{i}(E-F)).
\end{array}
\]
\end{itemize}
\end{example}

\bigskip


\end{document}